%% file: main.tex
\documentclass[a4paper, 12pt]{amsart}
\input{preamble}

\theoremstyle{plain}
\declaretheorem[name=Theorem,sibling=theorem,numberwithin=section]{theorem}
\declaretheorem[name=Theorem,sibling=theorem*,numbered=no]{theorem*}
\declaretheorem[name=Proposition,sibling=theorem]{proposition}
\declaretheorem[name=Lemma,sibling=theorem]{lemma}
\declaretheorem[name=Corollary,sibling=theorem]{corollary}

\theoremstyle{definition}
\declaretheorem[name=Definition, sibling=theorem]{definition}
\declaretheorem[name=Remark,sibling=theorem]{remark}
\declaretheorem[name=Example,sibling=theorem]{example}

\title[Constructing Thompson representatives]{Constructing Thompson representatives via pointed links}

\author{Susanna Terron}

\address{School of Mathematics and Statistics, University of Glasgow, Glasgow G12 8QQ, UK}

\email{s.terron.1@research.gla.ac.uk}

\begin{document}

\begin{abstract}
    We extend Jones' construction to obtain a surjective map from the Brown--Thompson group $F_3$ to the set of pointed links up to pointed isotopy. We then introduce an operation on~$F_3$, and use it to define a new monoid $(F_3, \diamond)$, called the central monoid. Using the extended version of Jones' construction, we obtain a surjective monoid homomorphism from the central monoid to the monoid of pointed links with connected sum. This allows us to introduce a standard form for connected sum representatives in~$F_3$, and we extend this construction to a certain family of links by defining disjoint union and linking moves on $F_3$.
\end{abstract}

\maketitle

\input{introduction}

\input{preliminaries}

\input{pointed_links}

\input{monoid}

\input{connected_sum}

\input{link_rep}

\bibliographystyle{amsxport}
\bibliography{refs}

\end{document}

%% file: preamble.tex
\usepackage{amssymb, amsthm, amsmath, amsfonts, thmtools, yhmath, bm, stmaryrd, mathrsfs}
\usepackage{enumerate,enumitem}
\usepackage{mathtools,microtype,braket}
\usepackage{stmaryrd}
\usepackage[alphabetic]{amsrefs}
\usepackage{tikz-cd}
\usepackage{faktor}
\usepackage[framemethod=tikz]{mdframed}
\usepackage{quiver}
\usepackage[colorlinks=true, linkcolor=black, citecolor=black]{hyperref}
\usepackage[noabbrev,capitalise]{cleveref}

\usepackage{graphicx}
\usepackage{subcaption}
\usepackage{caption}
\usepackage{listings}
\usepackage{listingsutf8}
\usepackage[dvipsnames]{xcolor}
\usepackage{float}
\usepackage{wrapfig}

\newcommand{\fakeenv}{} 
\newenvironment{restate}[2]  
{
  \renewcommand{\fakeenv}{#2}   
                                
  \theoremstyle{plain}
  \newtheorem*{\fakeenv}{#1~\ref{#2}} 
  \begin{\fakeenv}  
}
{ \end{\fakeenv} }

 \subjclass[2020]{
        57K10,
        20F65}
    \keywords{Thompson's group, Alexander theorem, connected sum, pointed links}

\newcommand{\N}{\mathbb{N}}

\newcommand{\Z}{\mathbb{Z}}

\newcommand{\cK}{\mathcal{K}} 
\newcommand{\cL}{\mathcal{L}}

\DeclareMathOperator{\Aut}{Aut}

\DeclareMathOperator{\id}{id}

\DeclareMathOperator{\lk}{lk}   



%% file: introduction.tex
\section{Introduction}

Inspired by questions arising from conformal field theory, Vaughan Jones introduced a construction associating links to elements of Thompson's group $F$ \cite{Jones2017} and its generalisation $F_3$ \cite{Jones2019}. In analogy to the theory relating braid groups and knot theory, this construction gives a map $F \to \mathsf{Links}$ similar to the closure map~$\bigsqcup_n B_n \to \mathsf{Links}$. In the case of braids, surjectivity of the map is guaranteed by the Alexander Theorem, and a set of moves sufficient to establish when two braids give the same link is described by the Markov Theorem. 

In his paper Jones proved an Alexander type theorem showing that all links can be obtained from an element of $F$ \cite[Theorem 5.3.1]{Jones2017}, proving that the map above is surjective. The question of when two group elements result in the same link is still unanswered. Maintaining the parallel with braid groups, answering this question would result in a Markov type theorem for either $F$ or $F_3$. We call these elements Thompson representatives of the link. 

Various authors have explored the connection between Thompson's group and knot theory, see for example \cite{aiello2020}, \cite{AielloNagnibeda2022}, \cite{AielloBaader2024}, \cite{liles2024annularlinksthompsonsgroup},  \cite{krushkal2024thompsonsgroupftangles} and \cite{RushilSweeney}. The aim of this paper is to introduce a standard form for representatives of certain links in the Brown--Thompson group~$F_3$ by defining operations on the group that correspond to operations on links. 

We first extend Jones' construction to a map $\cL_*$ with image in the set of pointed links up to pointed isotopy, $\mathsf{Links}_*$, and prove the following theorem.
\renewcommand{\thetheorem}{\Alph{theorem}}
\begin{theorem}
\label{thm:surjective on pointed}
    The map $\cL_*\colon F_3 \to \mathsf{Links}_*$ is surjective. 
\end{theorem}

We then consider connected sum of knots, which leads to the definition of the central monoid $(F_3, \diamond)$, whose operation can be rewritten in terms of the group operation of $F_3$. In particular, the operation $\diamond$ is a modified version of the attaching operations defined by Y. Kodama and A. Takano \cite{kodama2023alexanderstheoremstabilizersubgroups}. The monoid structure on $F_3$ turns $\cL_*$ into a surjective monoid homomorphism, with image the monoid of pointed links with operation given by connected sum on the distinguished components. In particular, we obtain the following result.

\begin{theorem}
\label{thm:monoid diagram}
    The following is a commutative diagram of surjective monoid homomorphisms.
    \begin{center}
    \begin{tikzcd}
    (F_3, \diamond) \arrow[rr, "\cL_*", two heads] \arrow[rrdd, "\cK"', two heads] &  & (\mathsf{Links_*}, \#_*) \arrow[dd, "U", two heads] \\
                                     &  &                   \\
                                     &  & (\mathsf{Knots}, \#)                
    \end{tikzcd}    
    \end{center}
\end{theorem}

Using the operation $\diamond$ one obtains a standard form for knot representatives up to prime decomposition and, as a consequence, possible moves for a Markov type theorem for $F_3$. A notion of standard form can then be given for a certain family of links by extending this construction via operations that correspond to disjoint union and linking components. Specifically, we obtain representatives for all tree links (see \cref{def:tree links}), which are defined by associating a link to a given labelled finite tree by interpreting vertices as components and labelled edges as linkings.

\begin{theorem}
\label{thm:tree rep}
    For all tree links one can construct a Thompson representative, up to prime components, using $-\diamond_{\alpha}-$, $-\sqcup -$ and the linking moves.
\end{theorem}

\subsection*{Overview}

In \cref{sec:preliminaries} we recall some background on Thompson's group $F$ and describe Jones' algorithm. We then extend Jones' construction in \cref{sec:pointed links}, obtaining a surjective map into $\mathsf{Links}_*$. We introduce the central monoid in \cref{sec: central monoid}, defining the operation $\diamond$ and comparing it to the group operation in $F_3$. We then discuss a generalisation of $\diamond$. In \cref{sec: connected sum} we translate the results of the previous section into low-dimensional topology, obtaining a standard form for knot representatives up to their prime components, and then extend this result in \cref{sec:link representatives} by defining disjoint union and the linking moves, and proving that, via these moves, one obtains representatives for all tree links.

\subsection*{Aknowledgements}
I would like to thank my supervisors Rachael Boyd and Mike Whittaker for their support and extremely useful insights. I am supported by the Engineering and Physical Sciences Research Council [grant number EP/Y035232/1], Centre for Doctoral Training in Algebra, Geometry and Quantum Fields (AGQ).

\renewcommand{\thetheorem}{\thesection.\arabic{theorem}}

%% file: preliminaries.tex
\section{Preliminaries on Thompson's group}
\label{sec:preliminaries}
We first introduce the basic notions needed throughout. We recall the definition of Thompson's group $F$, its generalisation $F_3$, and explain Jones' algorithm (see \cite{Jones2017}, \cite{Jones2019}). For a thorough overview on Thompson's group we refer the reader to \cite{CannonFloydParry} and \cite{BelkThesis}. 

\subsection{Thompson's group F}
Thompson's group $F$ is defined as the group of all piecewise-linear homeomorphisms of the unit interval with breakpoints having dyadic rational coordinates and slopes powers of~$2$. The group operation is given by function composition, with the convention $f\cdot g=g\circ f$. In particular, the elements of $F$ correspond to dyadic rearrangements of $[0,1]$, i.e.~order preserving piecewise-linear homeomorphisms between dyadic subdivisions with the same number of cuts. Hence, an element $f\in F$ can be represented by drawing the corresponding domain and range dyadic subdivision.  

\begin{figure}[H]
    \centering
    \begin{tikzpicture}
        \node[anchor=south west,inner sep=0] at (0.5,0.3){\includegraphics[width=0.27\textwidth]{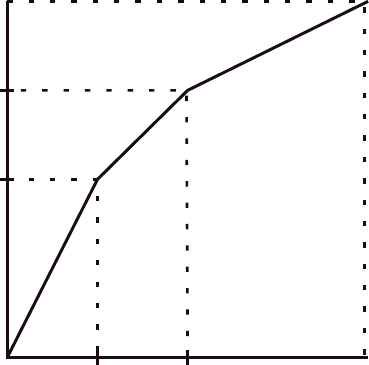}};
        \node at (0,0) {$0$};
        \node at (1.3,0) {$1/4$};
        \node at (2.3,0) {$1/2$};
        \node at (4,0) {$1$};
        \node at (0,2) {$1/2$};
        \node at (0,2.9) {$3/4$};
        \node at (0,3.7) {$1$};
        \node at (2,3.2) {$f$};
        \node at (5,2) {$\leftrightsquigarrow$};
        \node[anchor=south west,inner sep=0] at (6.5,1.5){\includegraphics[width=0.27\textwidth]{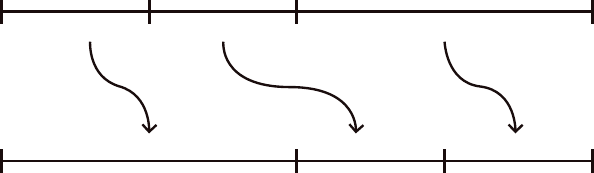}};
        \node at (6,2) {$f\colon$};
        \node at (6.5,2.8) {$0$};
        \node at (7.3,2.8) {$1/4$};
        \node at (8.2,2.8) {$1/2$};
        \node at (9.9,2.8) {$1$};
        \node at (6.5,1.2) {$0$};
        \node at (8.2,1.3) {$1/2$};
        \node at (9.1,1.3) {$3/4$};
        \node at (9.9,1.3) {$1$};
    \end{tikzpicture}
\end{figure}

When using this pictorial representation, we will often omit the arrows, as they are order preserving. Function composition corresponds to stacking intervals, and is read from top to bottom.

\begin{figure}[H]
    \centering
    \begin{tikzpicture}
        \node[anchor=south west,inner sep=0] at (1.5,0.3){\includegraphics[width=0.3\textwidth]{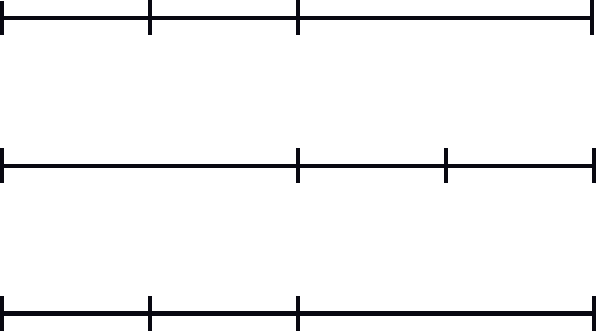}};
        \node at (0,1.5) {$f\cdot g=$};
        \node at (1.1,1.8) {$f$};
        \node at (1.1,0.8) {$g$};
        \node at (6,1.5) {$=id$};
    \end{tikzpicture}
\end{figure}

The group $F$ has presentation 
\[
    F=\langle x_0,x_1,\dots | x_nx_k=x_kx_{n+1} \forall k<n \rangle.
\]
Note that the elements $x_0$ and $x_1$ generate $F$, so we have a finite presentation
\[
    F=\langle x_0, x_1 | x_2x_1=x_1x_3, x_3x_1=x_1x_4 \rangle,
\]
where $x_n=(x_1)^{x_0^{n-1}}$. Pictorial representations of $x_0$ and $x_1$ are shown in \cref{example:product in F}.
Moreover, every element $g\in F$ can be expressed uniquely in a normal form
\[
    g=x_0^{a_0}x_1^{a_1}\dots x_n^{a_n}x_n^{-b_n}\dots x_0^{-b_0},
\]
where $a_0, \dots, a_n,b_0, \dots, b_n \in \N$ (exactly one of $a_n$ and $b_n$ is nonzero) and, if $a_i\neq 0$ and $b_i\neq 0$, then $a_{i+1}\neq 0$ or $b_{i+1}\neq 0$.

Any $g\in F$ can be represented as a pair of binary rooted trees~$T_+$ and~$T_-$, with the same number of leaves. Here the upper tree $T_+$ corresponds to the domain and the lower (flipped) tree $T_-$ to the range, by identifying the leaves of each tree with the intervals of the corresponding dyadic subdivision, as in \cref{fig:x0 trees}.

\begin{figure}[H]
    \centering
    \begin{tikzpicture}
        \node[anchor=south west,inner sep=0] at (0,0){\includegraphics[width=0.17\textwidth]{img/intro/ex_x0.pdf}};
        \node at (3.1,1.1) {$\leftrightsquigarrow$};
        \node[anchor=south west,inner sep=0] at (4,0.7){\includegraphics[width=0.27\textwidth]{img/intro/x0_intervals.pdf}};
        \node at (4,2) {$0$};
        \node at (4.8,2) {$1/4$};
        \node at (5.8,2) {$1/2$};
        \node at (7.4,2) {$1$};
        \node at (4,0.4) {$0$};
        \node at (5.8,0.5) {$1/2$};
        \node at (6.6,0.5) {$3/4$};
        \node at (7.4,0.5) {$1$};
        \node at (8.2,1.1) {$\leftrightsquigarrow$};
        \node[anchor=south west,inner sep=0] at (9,0){\includegraphics[width=0.17\textwidth]{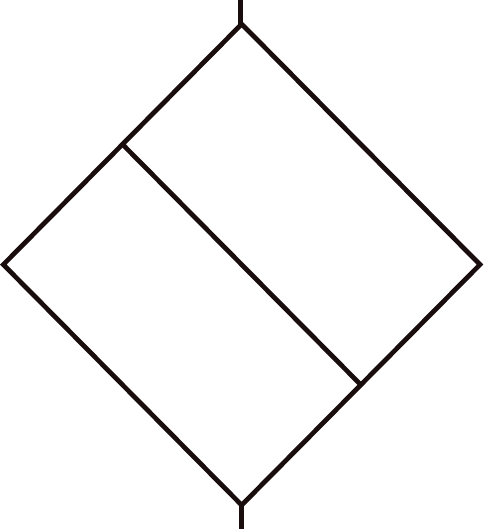}};
    \end{tikzpicture}
    \caption{The element $x_0\in F$, its dyadic subdivision, and the corresponding pair of trees.}
    \label{fig:x0 trees}
\end{figure}

We say that two pairs of binary trees are equivalent if and only if they differ up to adding or removing carets, as in \cref{fig:caret reduction}.

\begin{figure}[H]
    \centering
    \begin{tikzpicture}
        \node[anchor=south west,inner sep=0] at (0,0){\includegraphics[width=0.23\textwidth]{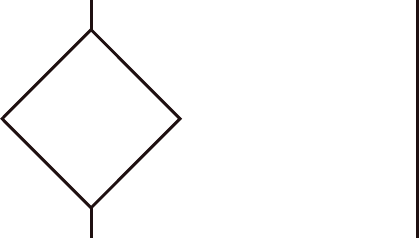}};
        \node at (2,0.8) {$\leftrightarrow$};
    \end{tikzpicture}
    \caption{Equivalence relation on binary trees. The left hand pair of trees is known as a caret.}
    \label{fig:caret reduction}
\end{figure}

A pair $(T_+, T_-)$ is said to be reduced if there are no carets. Each equivalence class has a unique reduced representative, and the correspondence between elements of $F$ and pairs of binary trees is unique up to considering reduced pairs of trees. 

The Guba--Sapir--Belk (GSB) algorithm (\cite{GubaSapir97}, \cite{BelkThesis}) provides rules to represent the group product via moves on tree pairs. Consider $f,g\in F$, and their corresponding pairs of reduced binary trees~$(T_+,T_-)$ and $(S_+,S_-)$, then the product $f\cdot g$ is represented by the pair of binary trees obtained by positioning $(T_+,T_-)$ above $(S_+,S_-)$, joining the root of $T_-$ with that of $S_+$, and repeatedly applying caret reduction and the move

\begin{figure}[H]
    \centering
    \begin{tikzpicture}
        \node[anchor=south west,inner sep=0] at (0,0){\includegraphics[width=0.23\textwidth]{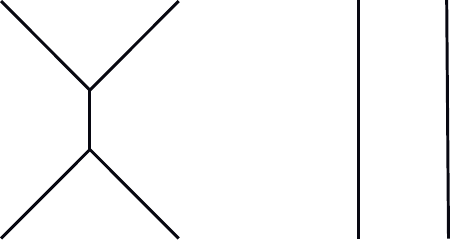}};
        \node at (1.7,0.8) {$\to$};
        \node at (4,0.8) {.};
    \end{tikzpicture}
\end{figure}

\begin{example}
\label{example:product in F}
    Consider $x_0,x_1\in F$, with reduced pairs of trees  
    \begin{figure}[H]
        \centering
        \begin{tikzpicture}
            \node at (0.2,1) {$x_0=$};
            \node[anchor=south west,inner sep=0] at (1,0){\includegraphics[width=0.15\textwidth]{img/intro/tree_x0.pdf}};
            \node at (5.2,1) {$x_1=$};
            \node[anchor=south west,inner sep=0] at (6,0){\includegraphics[width=0.15\textwidth]{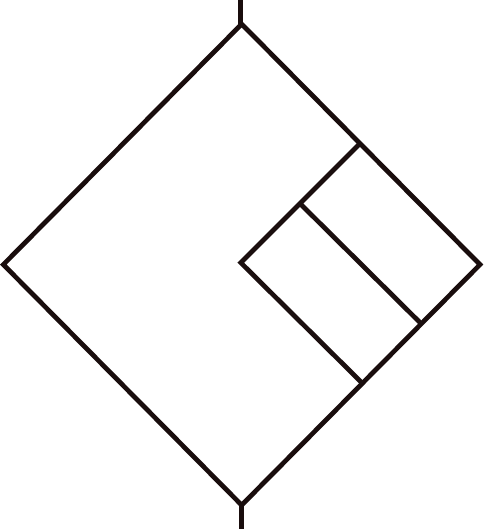}};
        \end{tikzpicture}
    \end{figure}
    Applying the GSB algorithm we obtain $x_0\cdot x_1$:
    \begin{figure}[H]
        \centering
        \begin{tikzpicture}
            \node[anchor=south west,inner sep=0] at (0,0){\includegraphics[width=0.15\textwidth]{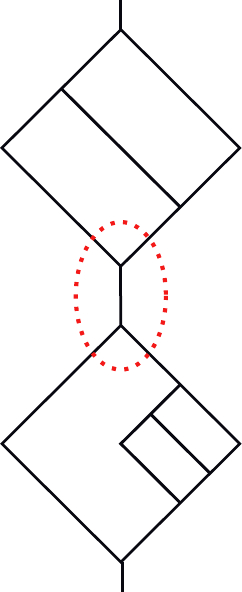}};
            \node at (2.5,2) {$\rightsquigarrow$};
            \node[anchor=south west,inner sep=0] at (3,0){\includegraphics[width=0.15\textwidth]{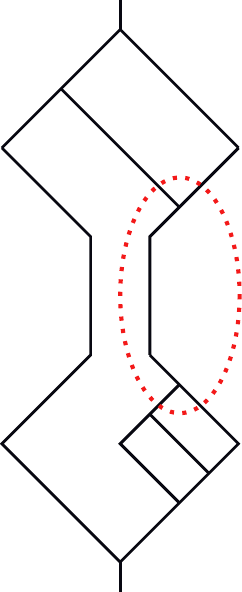}};
            \node at (5.5,2) {$\rightsquigarrow$};
            \node[anchor=south west,inner sep=0] at (6,0.3){\includegraphics[width=0.15\textwidth]{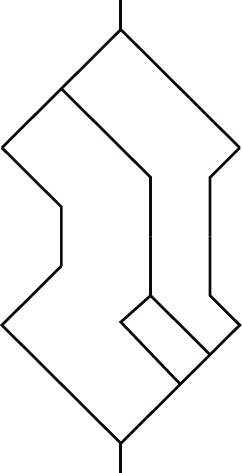}};
            \node at (8.5,2) {$\rightsquigarrow$};
            \node[anchor=south west,inner sep=0] at (9,0.8){\includegraphics[width=0.15\textwidth]{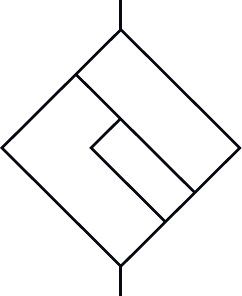}};
            \node at (11.5,2) {.};
        \end{tikzpicture}
    \end{figure}
\end{example}

\subsection{Brown--Thompson groups}
Thompson's group $F$ generalises to Brown--Thompson groups $F_k$, $k\geq 2$, corresponding to $k$-ary rearrangements of the interval. For each $F_k$, we have the following presentation:
\[
    F_k=\langle y_0, y_1, \dots | y_ny_m=y_my_{n+k-1} \forall m<n \rangle.
\]
Note that $F_2=F$, and the elements $y_0, \dots, y_{k-1}$ generate $F_k$ for all $k\geq 2$, giving the following finite presentation (see \cite[Section 2.1]{KodamaTakano24}) 
\begin{equation*}
   F_k= \left\langle
    y_0, \dotsc, y_{n-1}
    \middle\vert
    \begin{alignedat}{3}
        y_k^{y_0} &= y_k^{y_i} &&\quad 1 \le i < k \le n-1,\\
        y_k^{y_0^2} &= y_k^{y_0y_i} &&\quad 1 \le i, k \le n-1 \text{ and } k - 1 \le i,\\
        y_1^{y_0^3} &= y_1^{y_0^2y_{n-1}}
    \end{alignedat}
    \right\rangle .
\end{equation*}

Similarly to $F$, the group $F_k$ can be represented by pairs of $k$-ary trees. In particular, we will be interested in the group $F_3$, whose generators are represented in \cref{fig:generators of F3}.

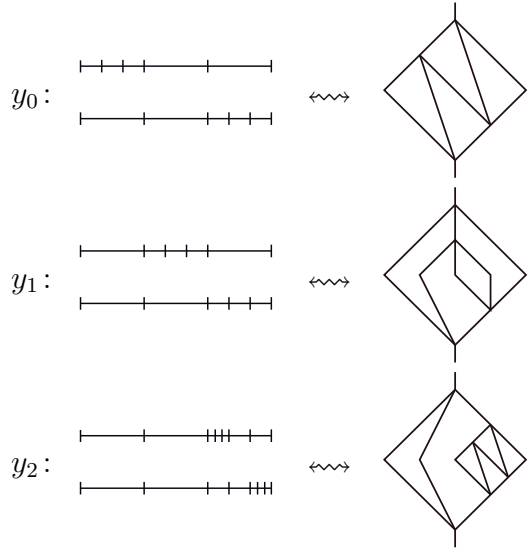
\begin{figure}[H]
    \centering
    \begin{tikzpicture}[scale=0.7]
        \node at (0,8.5) {$y_0\colon$};
        \node[anchor=south west,inner sep=0] at (0.8,8){\includegraphics[width=0.2\textwidth]{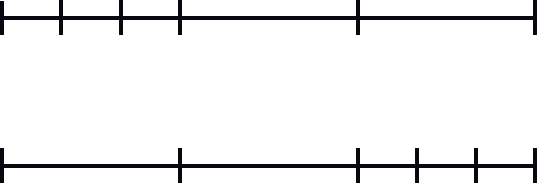}};
        \node at (5.5,8.5) {$\leftrightsquigarrow$};
        \node[anchor=south west,inner sep=0] at (6.5,7){\includegraphics[width=0.15\textwidth]{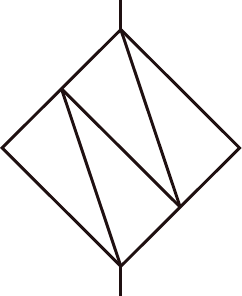}};
        \node at (0,5) {$y_1\colon$};
        \node[anchor=south west,inner sep=0] at (0.8,4.5){\includegraphics[width=0.2\textwidth]{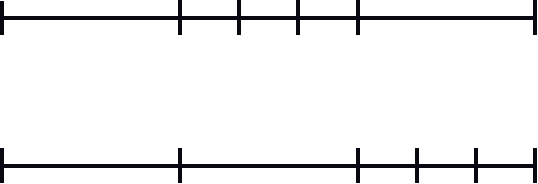}};
        \node at (5.5,5) {$\leftrightsquigarrow$};
        \node[anchor=south west,inner sep=0] at (6.5,3.5){\includegraphics[width=0.15\textwidth]{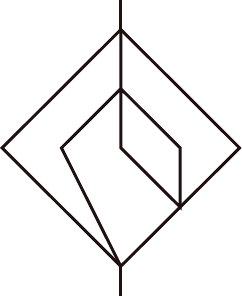}};
        \node at (0,1.5) {$y_2\colon$};
        \node[anchor=south west,inner sep=0] at (0.8,1){\includegraphics[width=0.2\textwidth]{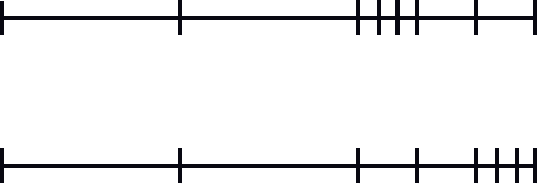}};
        \node at (5.5,1.5) {$\leftrightsquigarrow$};
        \node[anchor=south west,inner sep=0] at (6.5,0){\includegraphics[width=0.15\textwidth]{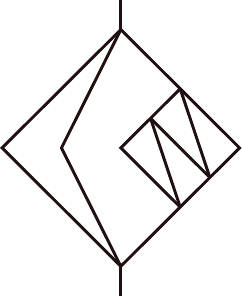}};
    \end{tikzpicture}
    \caption{The generators $y_0, y_1$ and $y_2$ of $F_3$ and the corresponding pairs of ternary trees.}
    \label{fig:generators of F3}
\end{figure}

In the case of $F_3$, caret reduction/inclusion is given by
\begin{figure}[H]
    \centering
    \begin{tikzpicture}
        \node[anchor=south west,inner sep=0] at (0,0){\includegraphics[width=0.23\textwidth]{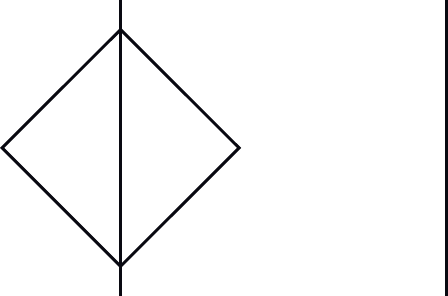}};
        \node at (2.2,1) {$\leftrightarrow$};
        \node at (4,1) {.};
    \end{tikzpicture}
\end{figure}
Analogously to the GSB algorithm, the group operation on $F_3$ is obtained by taking the corresponding pairs of ternary trees and repeatedly applying caret reduction and the move 

\begin{figure}[H]
    \centering
    \begin{tikzpicture}
        \node[anchor=south west,inner sep=0] at (0,0){\includegraphics[width=0.23\textwidth]{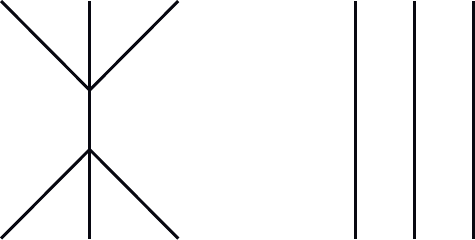}};
        \node at (1.6,0.8) {$\to$};
        \node at (4,0.8) {.};
    \end{tikzpicture}
\end{figure}

\subsection{Jones' algorithm}
In 2017 Vaughan Jones introduced a way to obtain knots and links starting from Thompson's group $F$ \cite{Jones2017}. Here we describe an equivalent construction presented in \cite{Jones2019}, that can be extended from $F$ to $F_3$. Let $\mathsf{Knots}$ and $\mathsf{Links}$ denote respectively the sets of knots and links up to isotopy. Take $g\in F$ and consider the image under the inclusion $F\hookrightarrow F_3$, obtained by sending $x_0$ to $y_0$ and~$x_1$ to $y_2$. At the level of trees, this corresponds to adding an edge to each vertex, as in \cref{fig:image of x0 in F3}.

\begin{figure}[H]
    \centering
        \begin{tikzpicture}
        \node[anchor=south west,inner sep=0] at (0,0){\includegraphics[width=0.18\textwidth]{img/intro/tree_x0.pdf}};
        \node at (4.2,1.3) {$\longrightarrow$};
        \node[anchor=south west,inner sep=0] at (6,0){\includegraphics[width=0.19\textwidth]{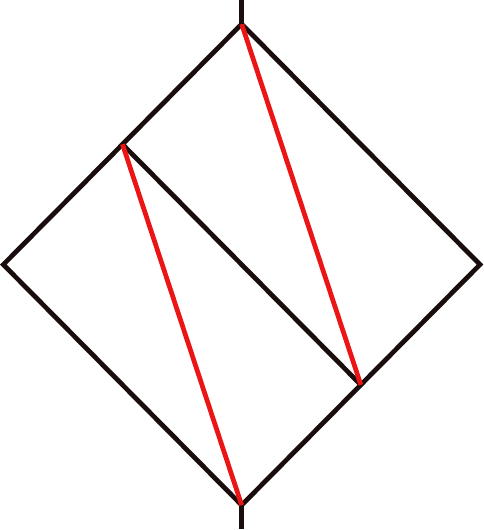}};
        \end{tikzpicture}
        \caption{The image of $x_0$ in $F_3$.}
        \label{fig:image of x0 in F3}
\end{figure}

To obtain a link from the image in $F_3$, we first replace each $4$-valent vertex with a crossing via the following rules.

\begin{figure}[H]
    \centering
        \begin{tikzpicture}
        \node[anchor=south west,inner sep=0] at (0,0){\includegraphics[width=0.1\textwidth]{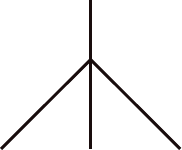}};
        \node at (2.7,0.3) {$\longrightarrow$};
        \node[anchor=south west,inner sep=0] at (4,0){\includegraphics[width=0.1\textwidth]{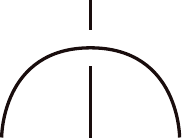}};
        \end{tikzpicture}
\end{figure}

\begin{figure}[H]
    \centering
        \begin{tikzpicture}
        \node[anchor=south west,inner sep=0] at (0,0){\includegraphics[width=0.1\textwidth]{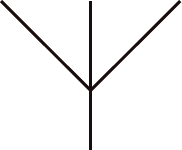}};
        \node at (2.7,0.5) {$\longrightarrow$};
        \node[anchor=south west,inner sep=0] at (4,0){\includegraphics[width=0.1\textwidth]{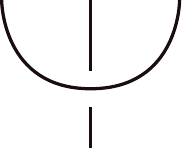}};
        \end{tikzpicture}
\end{figure}

Finally, we close the link by joining the two roots.

\begin{figure}[H]
    \centering
        \begin{tikzpicture}
        \node[anchor=south west,inner sep=0] at (0,0){\includegraphics[width=0.25\textwidth]{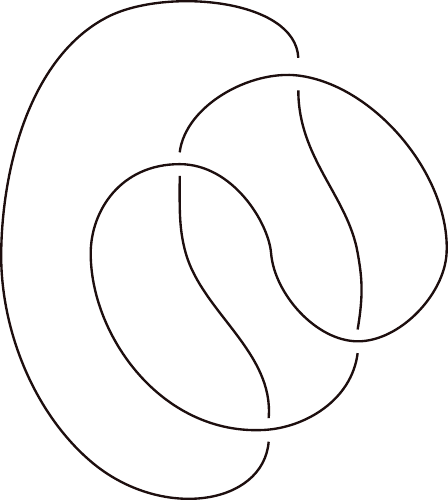}};
        \end{tikzpicture}
        \caption{The link obtained from $x_0\in F$ via Jones' algorithm. Notice this is the unknot.}
\end{figure}

Jones' algorithm gives a map $\cL\colon F \to \mathsf{Links}$, sending each element of $F$ to the isotopy class of the corresponding link. In analogy to the construction of links given by closing elements of the braid group, Jones proved an Alexander type theorem: the map $\cL$ is surjective.

\begin{theorem}\cite[Theorem 3.1]{Jones2019}
\label{thm:Alexander type thm}
    For any link $L$ there exists an element $g\in F$ such that $\cL(g)= L$.
\end{theorem}

Forgoing the first step in the construction, we can extend $\cL$ to a map on $F_3$ that, with an abuse of notation, will also be denoted by $\cL$.
\begin{definition}
    Define $\cL\colon F_3 \to \mathsf{Links}$ to be the map associating to each element $f\in F_3$ the isotopy class of the link obtained via Jones' algorithm.  
\end{definition}

Note that, since we are taking an extension of the map defined on $F$, \cref{thm:Alexander type thm} implies that $\cL\colon F_3 \to \mathsf{Links}$ is surjective.

%% file: pointed_links.tex
\section{Extension to pointed link}
\label{sec:pointed links}
Denote by $\mathsf{Links}_*$ the set of pointed links up to pointed isotopy, i.e.~the set of links with a marked component considered up to link isotopy that preserves the marking. In this section we extend Jones' result to obtain a surjective map into $\mathsf{Links}_*$.\\

Each element of $\mathsf{Links}_*$ corresponds to a pair $(L,C)$, with $L$ the isotopy class of the link and $C$ its marked component. Note that, for any~$f\in F_3$, Jones' construction gives us a preferred component in $\cL(f)$, corresponding to the component containing the closure of the roots of the trees. We call it the \textit{central component}. We can therefore extend $\cL$ to the following map.

\begin{definition}
    Define $\cL_*\colon F_3 \to \mathsf{Links}_*$ to be the map associating to each~$f\in F_3$ the class of the link $\cL(f)$ with marked component the central component. We call the image of an element $f\in F_3$ under $\cL_*$ a \textit{Thompson link}.
\end{definition}

Isotopic links may differ once one takes pointed isotopy, as the class depends on the marked component. For example, the links in \cref{fig:split links} correspond to the same class in $\mathsf{Links}$ but represent different classes in~$\mathsf{Links}_*$.

\begin{figure}[H]
    \centering
    \begin{subfigure}[b]{0.45\textwidth}
        \centering
        \includegraphics[width=0.7\textwidth]{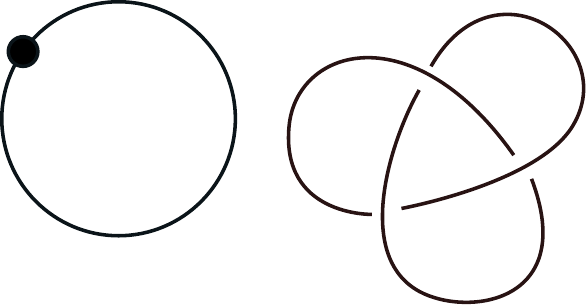}
    \end{subfigure}
    \;
    \begin{subfigure}[b]{0.45\textwidth}
        \centering
        \includegraphics[width=0.7\textwidth]{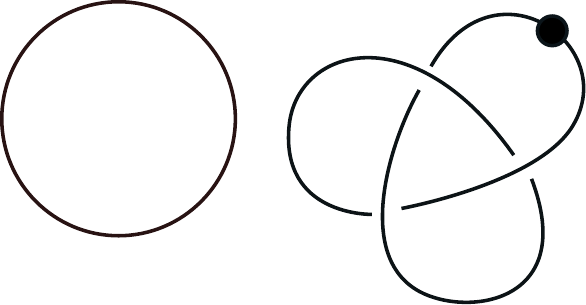}
    \end{subfigure}
    \caption{A split link with marked component the unknot (on the left) or the trefoil (on the right).}
    \label{fig:split links}
\end{figure}

We extend Jones' result to obtain the following theorem.
\begin{restate}{Theorem}{thm:surjective on pointed}
    The map $\cL_*\colon F_3 \to \mathsf{Links}_*$ is surjective. 
\end{restate}

To prove \cref{thm:surjective on pointed} we show that, given a representative $f\in F_3$ of a link $L$, we can modify $f$, without changing the isotopy class of the link, so that any component can be made to be the central component. To do so, we use the following two lemmas.

\begin{lemma}
\label{lemma:change central component}
    Consider $f\in F_3$, with reduced pair of trees $(T_+, T_-)$. Let $\lambda$ be the rightmost leaf of $T_+$, and denote by $C_{\lambda}$ the component of~$\cL(f)$ containing the image of $\lambda$. Then, there exists $f'\in F_3$ such that~$\cL(f)=\cL(f')$ and $C_{\lambda}$ is the central component of $\cL(f')$.
\end{lemma}

\begin{proof}
    Given $f\in F_3$, let $(T_+', T_-')$ be the pair of ternary trees obtained from $(T_+,T_-)$ as follows:
    
    \begin{figure}[H]
    \centering
    \begin{tikzpicture}[scale=1]
    \node[anchor=south west,inner sep=0] at (0.5,1.5){\includegraphics[width=0.3\textwidth]{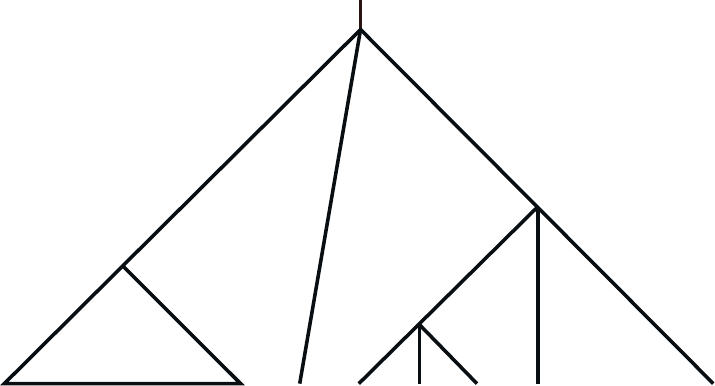}};
    \node at (0,2) {$T_+'\colon=$};
    \node at (1.2,1.7) {$T_+$};
    \node at (5,2) {and};
    \node[anchor=south west,inner sep=0] at (7,0.5){\includegraphics[width=0.3\textwidth]{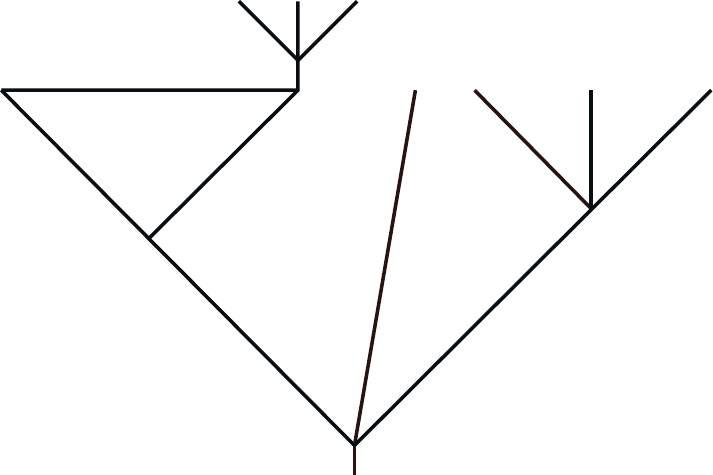}};
    \node at (6.5,2) {$T_-'\colon=$};
    \node at (7.8,2.2) {$T_-$};
    \node at (11,2) {.};
    \end{tikzpicture}
    \end{figure}
    
    Define $f'\in F_3$ as the element corresponding to the reduced pair of ternary trees $(T_+', T_-')$, i.e.~the element  
    \begin{figure}[H]
    \centering
    \begin{tikzpicture}[scale=0.7]
    \node[anchor=south west,inner sep=0] at (1,0){\includegraphics[width=0.3\textwidth]{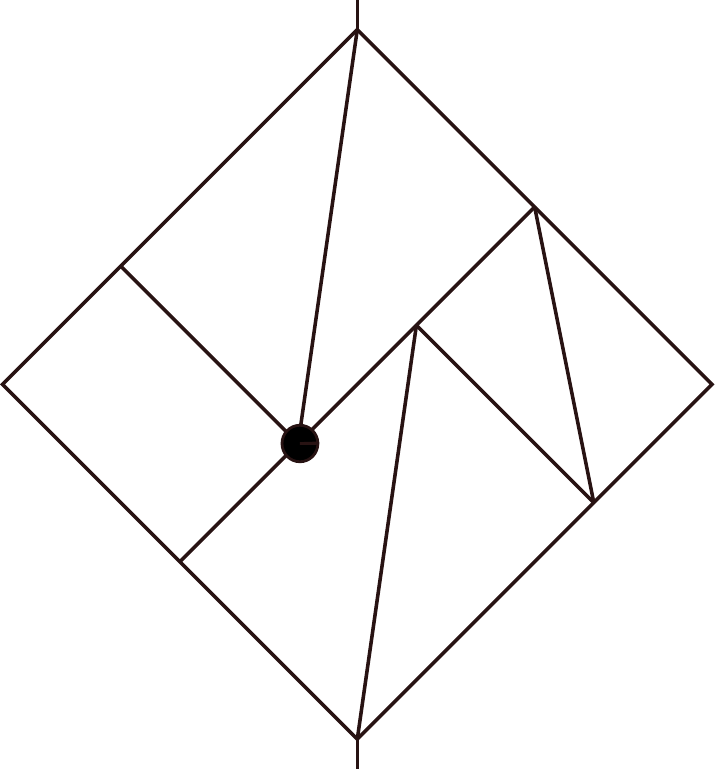}};
    \node at (0,2.9) {$f'=$};
    \node at (2,2.7) {$f$}; 
    \node at (3.3,1.8) {$\lambda$};
    \node at (6.7,2.9) {.};
    \end{tikzpicture}
    \end{figure}
    Note that the pair of trees of $f$ appears in $f'$, but the leaf $\lambda$ is now a vertex in $T_-'$ and not a leaf. Applying $\cL$ to $f'$, we obtain:
    \begin{figure}[H]
    \centering
    \begin{tikzpicture}[scale=0.7]
    \node[anchor=south west,inner sep=0] at (0.7,0){\includegraphics[width=0.25\textwidth]{img/surjectivity/surj1a.pdf}};
    \node at (0,2.4) {$f'=$};
    \node at (1.7,2.2) {$f$}; 
    \node at (2.9,1.7) {$\lambda$};
    \node at (6.7,2.3) {$\longrightarrow$};
    \node at (6.7,2.7) {$\cL$};
    \node[anchor=south west,inner sep=0] at (8,0){\includegraphics[width=0.27\textwidth]{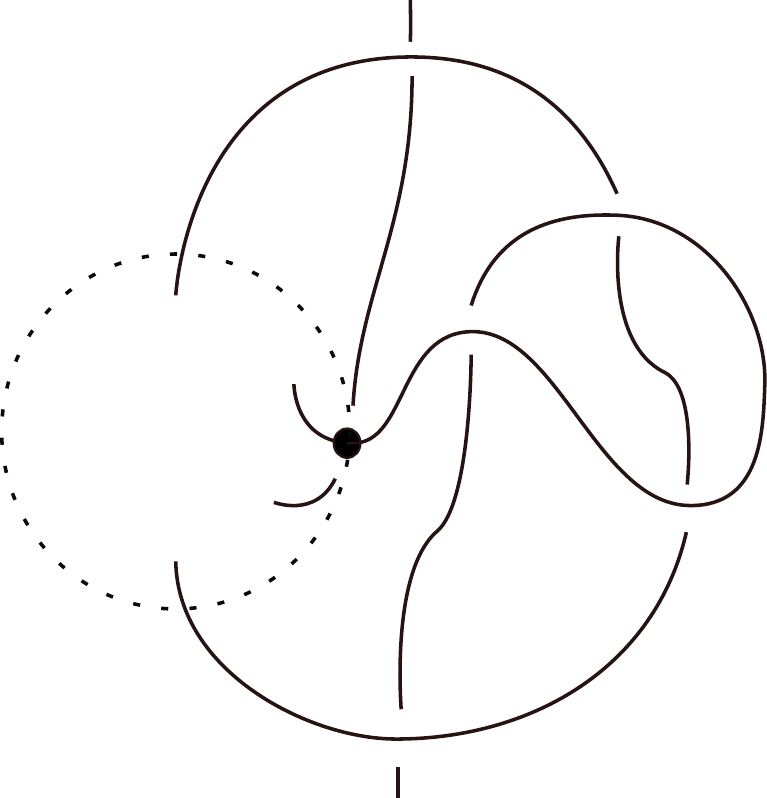}};
    \node at (9,2.5) {$\cL(f)$};
    \node at (13,3.7) {$C_{\lambda}$};
    \node at (14.1,2.4) {$=\cL(f')$.};
    \end{tikzpicture}
    \end{figure}

    Note that, via a sequence of Reidemeister I and II moves, the image of $f'$ is isotopic to that of $f$, i.e.~$\cL(f')=\cL(f)$, and $\cL(f')$ has central component $C_{\lambda}$. Hence, the element $f'$ satisfies the lemma.
\end{proof}

\begin{lemma}
\label{lemma:move leaf}
    Take $f\in F_3$ with reduced pair of trees $(T_+, T_-)$. Let $\alpha$ be any leaf in $T_+$ and denote by $C_{\alpha}$ the component of $\cL(f)$ containing the image of $\alpha$. Then, there exists $f'\in F_3$ with reduced pair of trees~$(T'_+, T'_-)$ such that $\cL(f')=\cL(f)$ and the component $C_{\alpha}$ passes through the rightmost leaf of $T'_+$.
\end{lemma}

\begin{proof}
    If $\alpha$ is the rightmost leaf of $T_+$ we conclude by taking $f'=f$. Suppose $\alpha$ is not the rightmost leaf of $T_+$, then there exists a leaf $\alpha'$ directly to the right of $\alpha$. Modify the trees $T_+$ and $T_-$ by adding vertices near the leaves $\alpha$ and $\alpha'$ as follows: 
    
    \begin{figure}[H]
    \centering
    \begin{tikzpicture}[scale=0.7]
    \node[anchor=south west,inner sep=0] at (0,0){\includegraphics[width=0.14\textwidth]{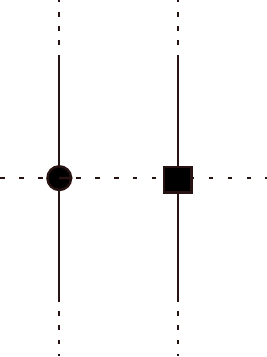}};
    \node at (0.2,2) {$\alpha$};
    \node at (2.2,2) {$\alpha'$};
    \node at (3.2,1.6) {$\rightsquigarrow$};
    \node[anchor=south west,inner sep=0] at (4,0){\includegraphics[width=0.25\textwidth]{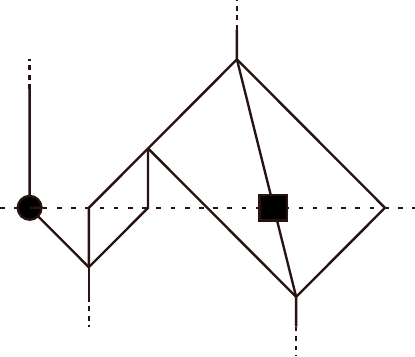}};
    \node at (4,2) {$\alpha$};
    \node at (6.5,2) {$\alpha'$};
    \node at (9.5,1.6) {.};
    \end{tikzpicture}
    \end{figure}

    Applying $\cL$ to the modified pair of trees gives the following local diagram: 
    
    \begin{figure}[H]
    \centering
    \begin{tikzpicture}[scale=0.7]
    \node[anchor=south west,inner sep=0] at (0,0){\includegraphics[width=0.25\textwidth]{img/surjectivity/surj2b.pdf}};
    \node at (5.3,1.5) {$\rightsquigarrow$};
    \node[anchor=south west,inner sep=0] at (6,0){\includegraphics[width=0.27\textwidth]{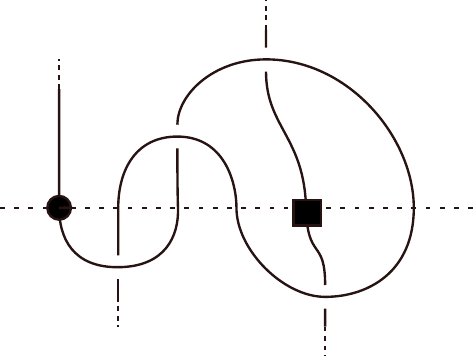}};
    \node at (10,3) {$C_{\alpha}$};
    \node at (11.5,1.6) {.};
    \end{tikzpicture}
    \end{figure}

    Observe that the diagram on the right can be modified via Reidemeister moves I and II to obtain the original diagram for $\cL(f)$. Hence, the modified link is in the isotopy class $\cL(f)$, but the component $C_{\alpha}$ passes through a leaf to the right of $\alpha'$.

    Define $f'\in F_3$ as the element obtained by iterating the above procedure until $C_{\alpha}$ passes through the rightmost leaf of the modified tree. Then $f'$ satisfies the lemma.
\end{proof}

We can now prove \cref{thm:surjective on pointed}.
\begin{proof}[Proof of \cref{thm:surjective on pointed}]
    Let $(L,C)\in\mathsf{Links_*}$ be a link with marked component $C$. By \cref{thm:Alexander type thm}, there exists $f\in F_3$ such that $\cL(f)=L$. If~$C$ is the central component of $\cL(f)$ we have that $\cL_*(f)=(L,C)$.

    Suppose $C$ is not the  central component of $\cL(f)$. Note that, Jones' algorithm guarantees that each component of $\cL(f)$ passes through at least one leaf of the upper tree $T_+$ of $f$. Hence, there exists a leaf $\alpha$ of $T_+$ such that $C$ passes through it. Define $f'\in F_3$ as the element obtained by applying \cref{lemma:move leaf} to $f$ and $\alpha$. Then, the element $f''\in F_3$ given by applying \cref{lemma:change central component} to $f'$ is such that $\cL(f'')=\cL(f')=\cL(f)$ and its central component is the component passing through $\alpha$, i.e.~it has central component $C$. Therefore $\cL_*(f'')=(L,C)$.
\end{proof}

Let $U_*\colon\mathsf{Links}_*\to\mathsf{Links}$ be the map obtained by forgetting the marking. We have the following commutative diagram of surjective maps:
\[
    \begin{tikzcd}
	&&& \mathsf{Links}_* \\
	\\
	{F_3} &&& \mathsf{Links}
	\arrow["U_*", two heads, from=1-4, to=3-4]
	\arrow["{\cL_*}", two heads, from=3-1, to=1-4]
	\arrow["\cL", two heads, from=3-1, to=3-4]
    \end{tikzcd},
\]
i.e.~$\cL=U_*\cL_*$.

%% file: monoid.tex
\section{The central monoid}
\label{sec: central monoid}
We now define the main algebraic object of the following sections, the \textit{central monoid}. Starting from the group $(F_3, \cdot)$, we introduce the operation $-\diamond -$ and consider the monoid $(F_3, \diamond)$. This operation is a modified version of the attaching operations of  \cite{kodama2023alexanderstheoremstabilizersubgroups}.\\

Consider $f\in F_3$ with associated pair of reduced trees $(T_+, T_-)$. Throughout this section we will always consider reduced pairs of trees unless otherwise specified. 
\begin{definition}
\label{def: central leaf}
    The \textit{central leaf} of $f$, denoted by $l(f)$, is the leaf of $T_+$ corresponding to the subinterval of $[0,1]$ containing $1/2$ in the triadic subdivision determined by $f$. We denote by $I_f$ the corresponding interval -- see for example \cref{fig:central leaf}. 
    \begin{figure}[H]
        \centering
        \includegraphics[width=0.3\linewidth]{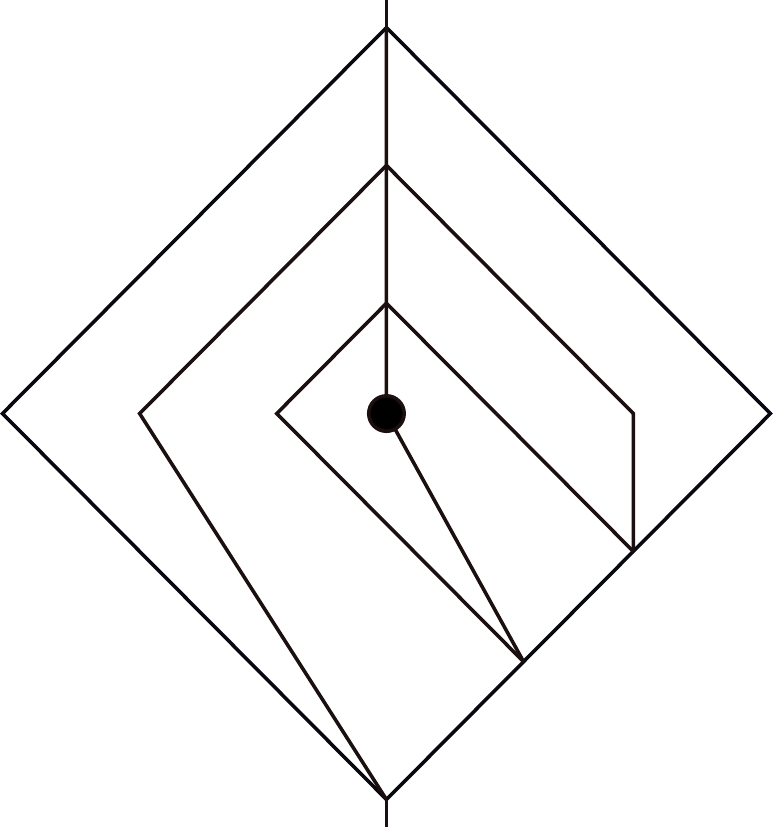}
        \caption{Example of central leaf, shown as marked vertex. Here $I_f=[13/27, 14/27]$.}
        \label{fig:central leaf}
    \end{figure}
\end{definition}

\begin{definition}
    Define $f\diamond g$ as the element of $F_3$ represented by the pair of trees obtained by attaching the pair of trees for $g$ in the position of the leaf $l(f)$, as illustrated in \cref{fig: example diamond}.
\end{definition}

\begin{figure}[H]
    \centering
    \begin{tikzpicture}
    \node[anchor=south west,inner sep=0] at (0,0.3){\includegraphics[width=0.25\textwidth]{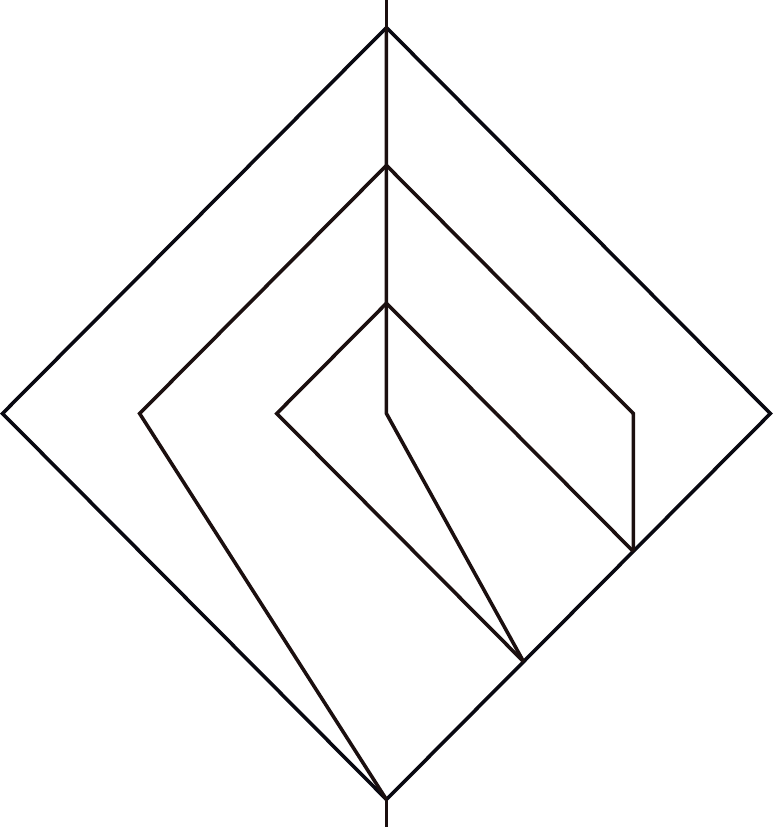}};
    \node at (3.6,2) {$\diamond$};
    \node[anchor=south west,inner sep=0] at (4,0.3){\includegraphics[width=0.25\textwidth]{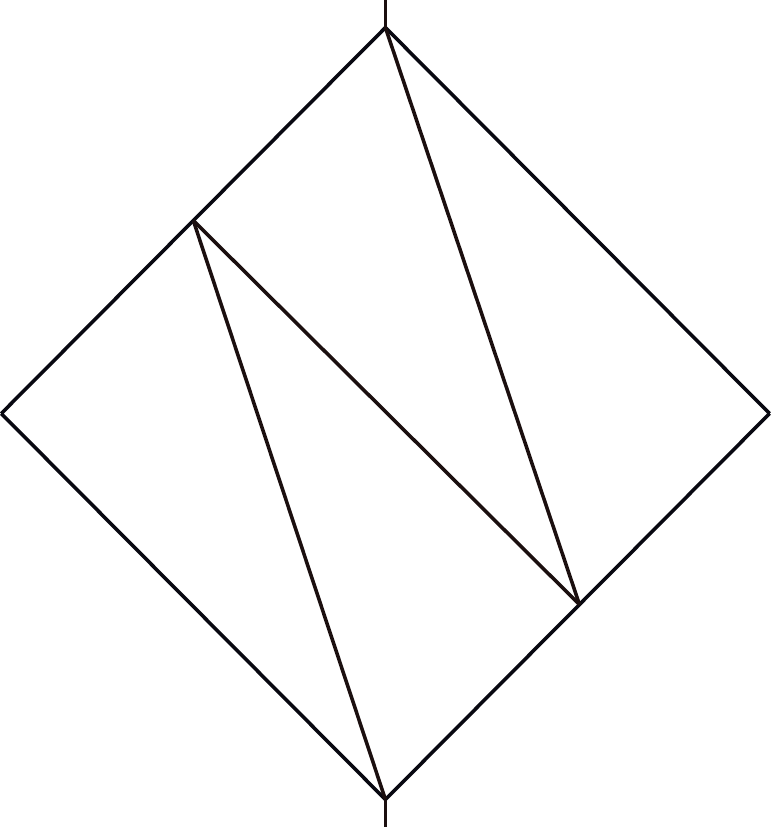}};
    \node at (8,2) {$=$};
    \node[anchor=south west,inner sep=0] at (8.5,0){\includegraphics[width=0.3\textwidth]{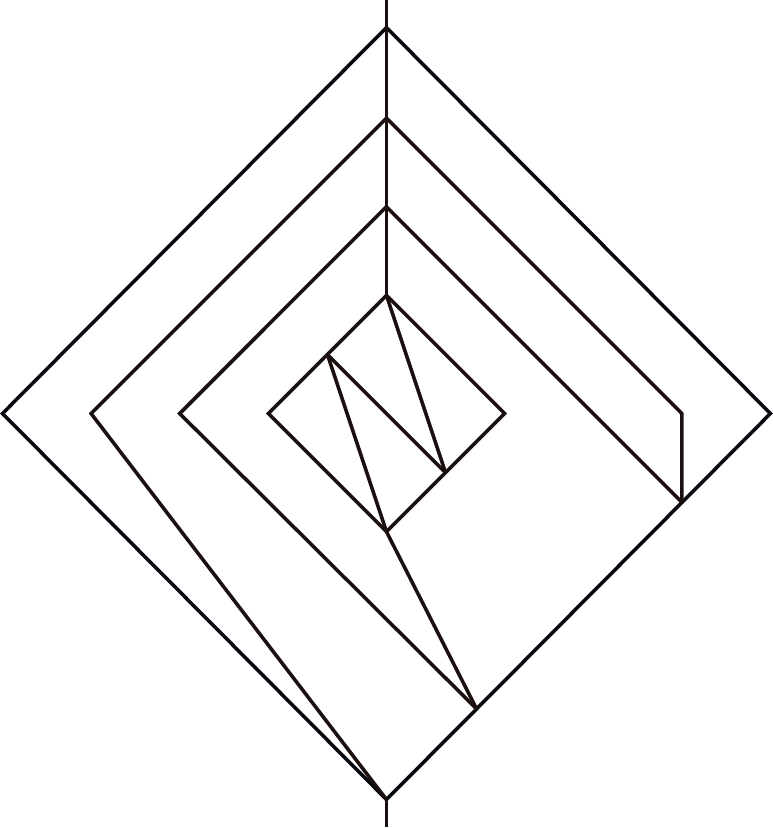}};
    \end{tikzpicture}
    \caption{An example of $f\diamond g \in F_3$.}
    \label{fig: example diamond}
\end{figure}

Note that $f\diamond g$ is well-defined as we are considering reduced pairs of trees. The map 
\[
    - \diamond -\colon F_3\times F_3 \longrightarrow F_3 
\]
defines an operation on $F_3$. For generic functions $f,g\in F_3$, we will use the following pictorial convention to represent the element $f\diamond g$.
\begin{figure}[H]
    \centering
    \begin{tikzpicture}[scale=0.7]
    \node[anchor=south west,inner sep=0] at (0,0){\includegraphics[width=0.3\textwidth]{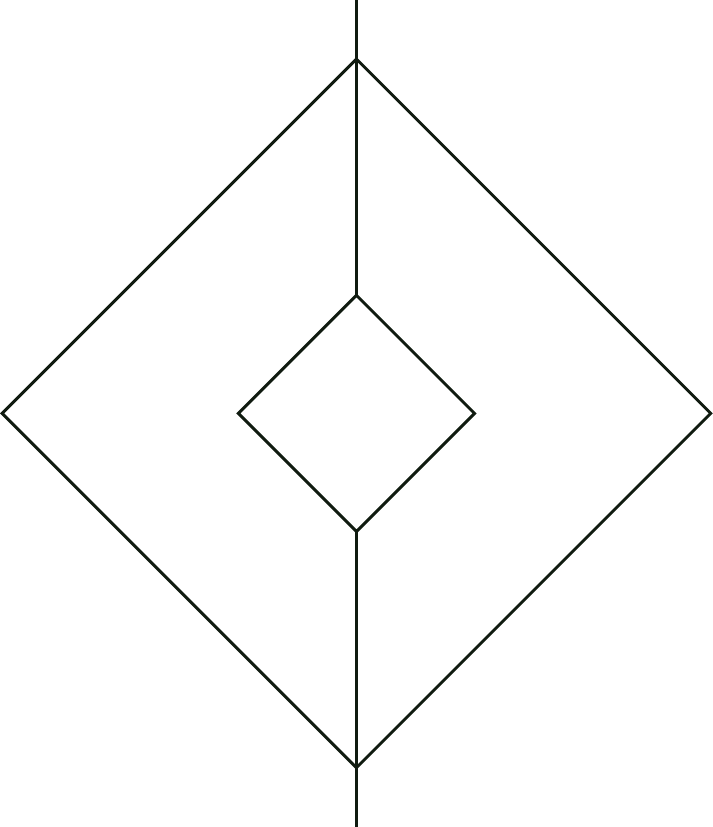}};
    \node at (2.7,3.1) {$g$};
    \node at (1.,3.3) {$f$};
    \end{tikzpicture}
    \label{fig:general fig of fdiamg}
\end{figure}

\begin{proposition}
\label{prop: diamond gives a monoid}
    The pair $(F_3, \diamond)$ is a monoid.
\end{proposition}
\begin{proof}
   Let $1_{F_3}$ denote the identity element in $F_3$. The reduced pair of trees associated to $1_{F_3}$ corresponds to the trivial tree, with no $4$-valent vertices. Therefore, for all $f\in F_3$ we have 
   \[
        f\diamond 1_{F_3}=1_{F_3}\diamond f=f \, .
   \]
   
   Note that for all $f,g,h\in F_3$ we have
   \[
   f\diamond (g\diamond h)=(f\diamond g)\diamond h,
   \]
    as shown in \cref{fig:associativity of diamond}, hence $-\diamond -$ satisfies associativity.   
\end{proof}
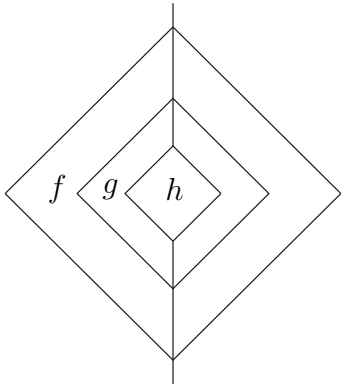
\begin{figure}[H]
    \centering
    \begin{tikzpicture}[scale=0.7]
    \node[anchor=south west,inner sep=0] at (0,0){\includegraphics[width=0.35\textwidth]{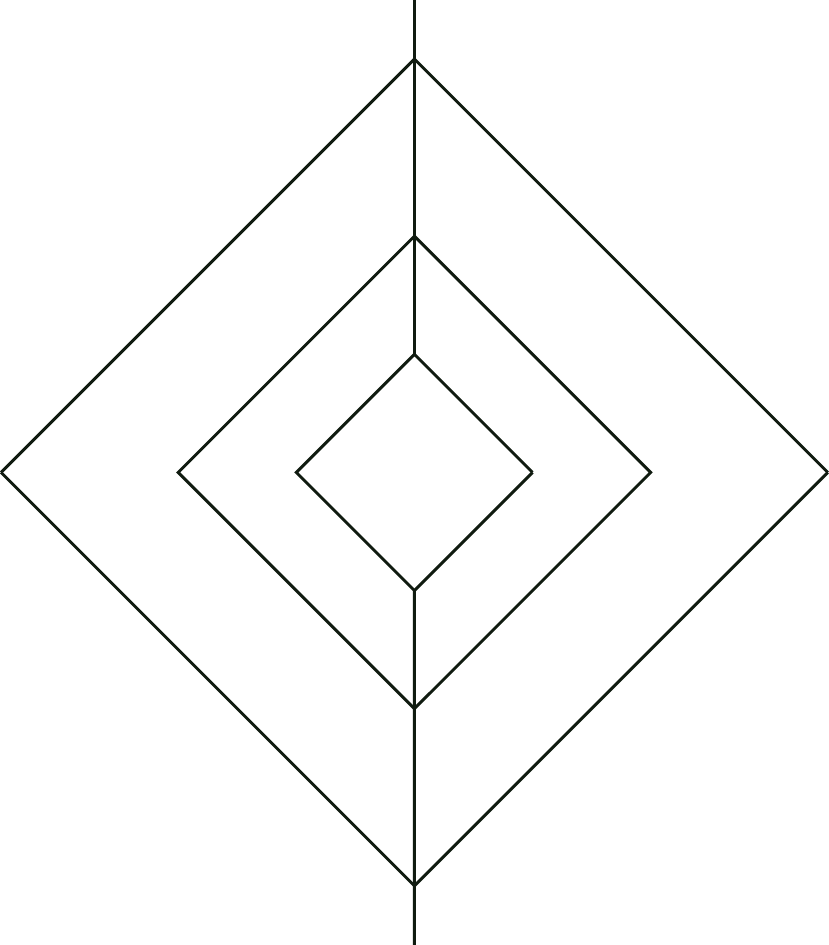}};
    \node at (1,3.7) {$f$};
    \node at (2,3.7) {$g$};
    \node at (3.2,3.7) {$h$};
    \end{tikzpicture}
    \caption{Pictorial representation of $f\diamond g\diamond h$.}
    \label{fig:associativity of diamond}
\end{figure}

\begin{definition}
    We call the monoid $(F_3, \diamond)$ the \textit{central monoid}.
\end{definition}

Observe that $(F_3,\diamond)$ has no invertible elements other than the identity. In fact, for any $f,g\in F_3$, the element $f\diamond g$ is represented by a pair of reduced trees with number of vertices given by the sum of the number of vertices of $f$ and $g$. Therefore, for any $f\neq 1_{F_3}$ there exists no $g\in F_3$ such that $f\diamond g=1_{F_3}$.

Specifically, $(F_3, \diamond)$ is a free infinitely generated monoid. The generators are the elements of $F_3$ such that the corresponding reduced pair of trees does not have a `diamond' around its central leaf, unless the element acts as the identity outside the diamond itself. Note that this does not exclude elements with diamonds around leaves other than the central one.

\begin{figure}[H]
    \centering
    \begin{subfigure}[b]{0.45\textwidth}
        \centering
        \includegraphics[width=0.55\textwidth]{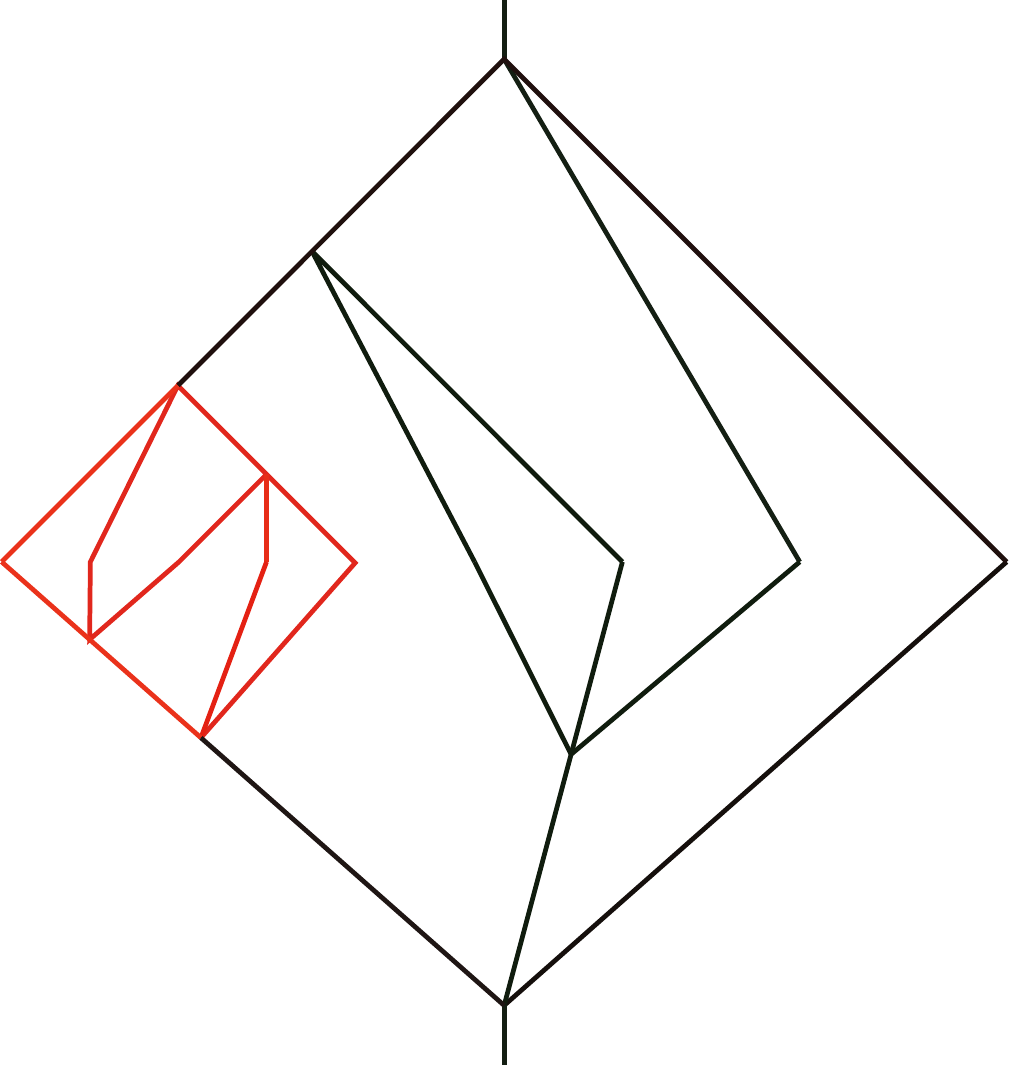}
        \caption*{A generator of $(F_3, \diamond)$.}
    \end{subfigure}
    \;
    \begin{subfigure}[b]{0.45\textwidth}
        \centering
        \includegraphics[width=0.55\textwidth]{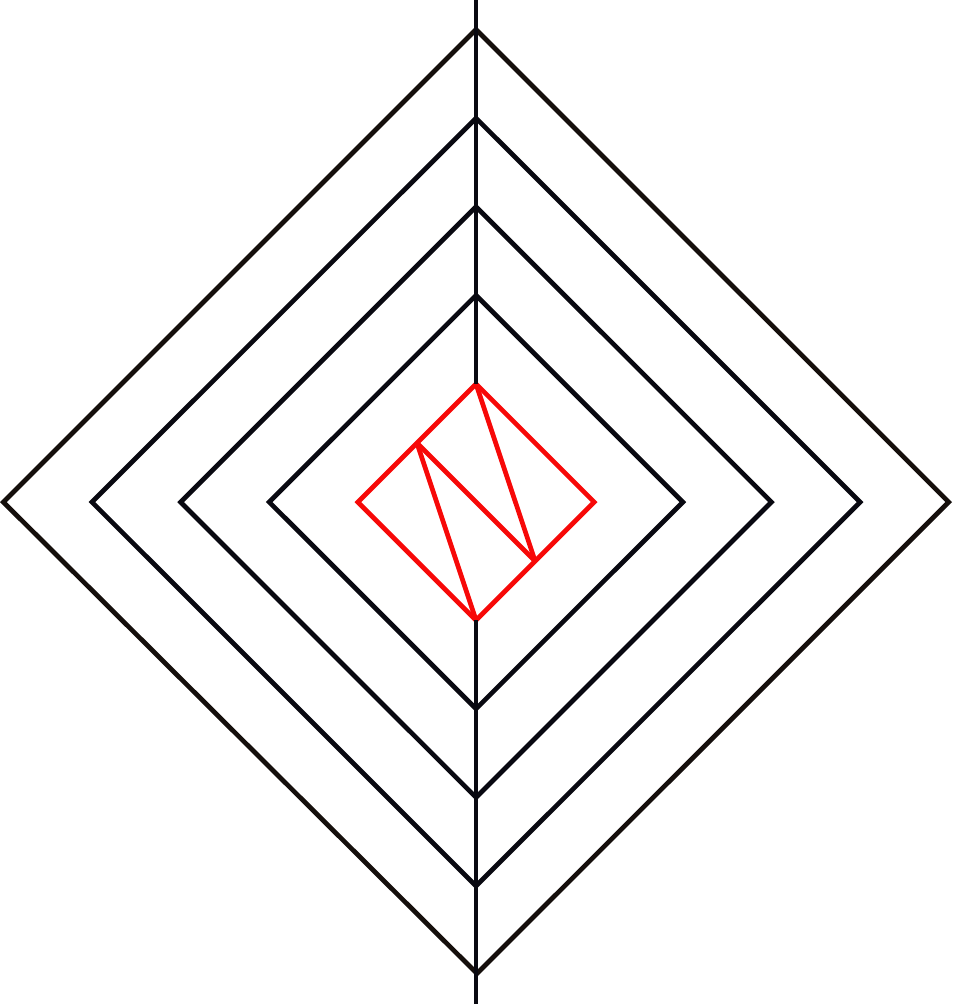}
        \caption*{A generator of $(F_3, \diamond)$.}
    \end{subfigure}
    \;
    \begin{subfigure}[b]{0.45\textwidth}
        \centering
        \includegraphics[width=0.55\textwidth]{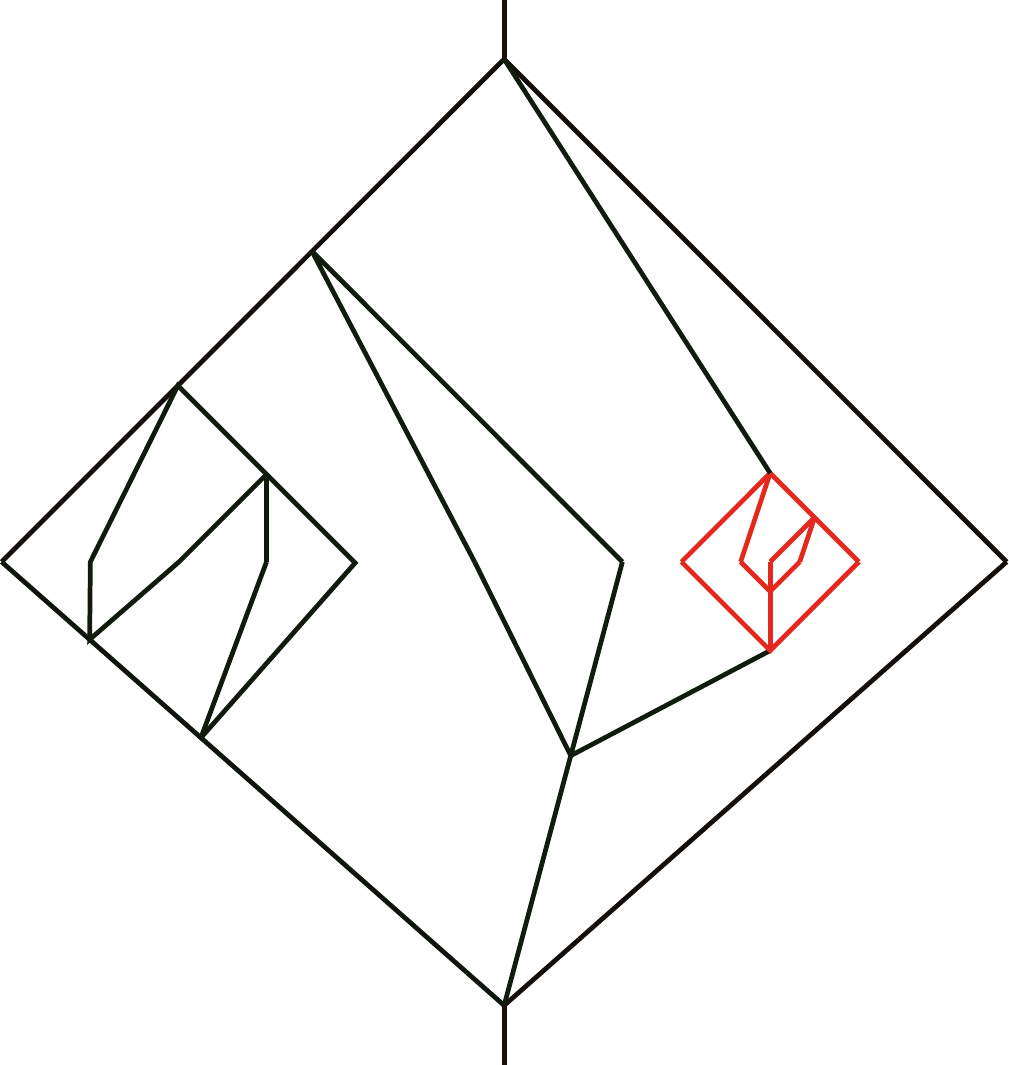}
        \caption*{Not a generator of $(F_3, \diamond)$.}
    \end{subfigure}
\end{figure}

Being free, $(F_3,\diamond)$ is a cancellative monoid, i.e.~for all $f,g,h\in F_3$ we have 
\[
    f\diamond g=f\diamond h  \Leftrightarrow g=h, \;\;\;\;\;\; g\diamond f=h\diamond f \Leftrightarrow g=h .
\]

\subsection{Comparing operations}
Consider the group $(F_3, \cdot)$. We compare the group operation with $-\diamond-$. It is in fact possible to express $f\diamond g$ in terms of composition of functions. This relation is observed in \cite[Remark 2.4]{kodama2023alexanderstheoremstabilizersubgroups} for the attaching operations.

Many of the results of the following section will not be related to later sections, but are included here as of independent interest.\\

Given a ternary tree $T$, we can describe the path from the root of~$T$ to any vertex by a sequence composed of $0$s, $1$s and $2$s using the following convention. The root of $T$ corresponds to the empty word $\emptyset$ and, any time the path meets a vertex, write $0$ to ``go left", $1$ to ``go straight" and $2$ to ``go right".

\begin{figure}[H]
    \centering
    \includegraphics[width=0.3\linewidth]{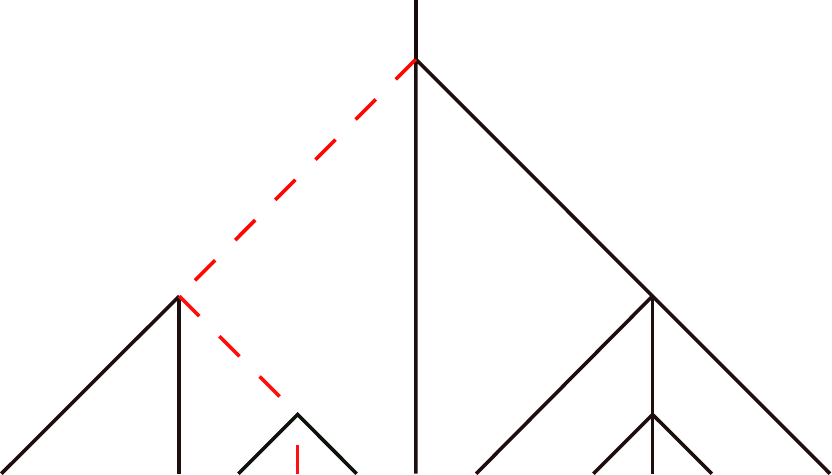}
    \caption{A tree $T$ with the path $\alpha=021$ in dashed lines.}
\end{figure}

Consider the set of finite sequences $\{0,1,2\}^*=\bigcup_n \{0,1,2\}^n$. Note that each vertex $v$ of $T$ is uniquely determined by the path going from the root to $v$, and, consequently, by a sequence $\alpha\in\{0,1,2\}^*$. We will therefore use $\alpha$ to denote the corresponding vertex of the tree, whenever it exists.

We call any finite sequence $\alpha\in\{0,1,2\}^*$ an \textit{address}, and use $|\alpha|$ to denote the length of $\alpha$. Endow addresses with the partial order $\preceq$ defined by 
\[
    \alpha\preceq\beta \Leftrightarrow \exists~\gamma\in\{0,1,2\}^* \text{ such that } \beta=\alpha\gamma,
\]
i.e.~$\alpha$ appears at the start of $\beta$.

Given a ternary tree $T$, let $V(T)$ be the set of all addresses corresponding to vertices of $T$. Fix $f\in F_3$ with pair of trees $(T_+, T_-)$. We use $V_+(f)$ and $V_-(f)$ to denote $V(T_+)$ and $V(T_-)$ respectively. 

Let $I_{\alpha}$ be the sub-interval of $I$ corresponding to the address~$\alpha$. Then, for any $f\in F_3$, we use $f(\alpha)$ to indicate the address in $V_-(f)$ corresponding to the interval $f(I_{\alpha})$.

\begin{figure}[H]
    \centering
    \includegraphics[width=0.35\linewidth]{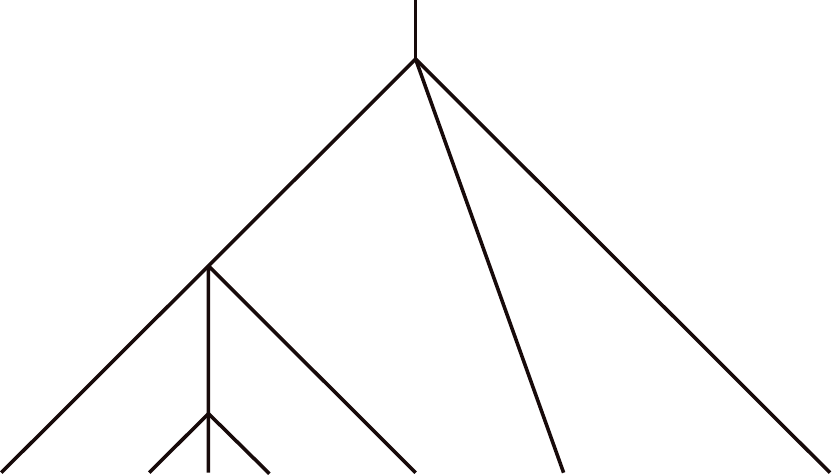}
    \caption{$V(T)=\{\emptyset, 0, 00, 01, 010, 011, 012, 02, 1, 2\}$.}
\end{figure}

Note that the central leaf $l(f)$ corresponds to the address in $V_+(f)$ given by the maximal sequence of $1$s. Specifically, the identity element has $l(1_{F_3})=\emptyset$ and, given $f,g\in F_3$ with $l(f)=1^n$ and $l(g)=1^m$, we have $l(f\diamond g)=1^{n+m}$ -- see \cref{fig:sum leaves}.
\begin{figure}[H]
    \centering
    \includegraphics[width=0.3\linewidth]{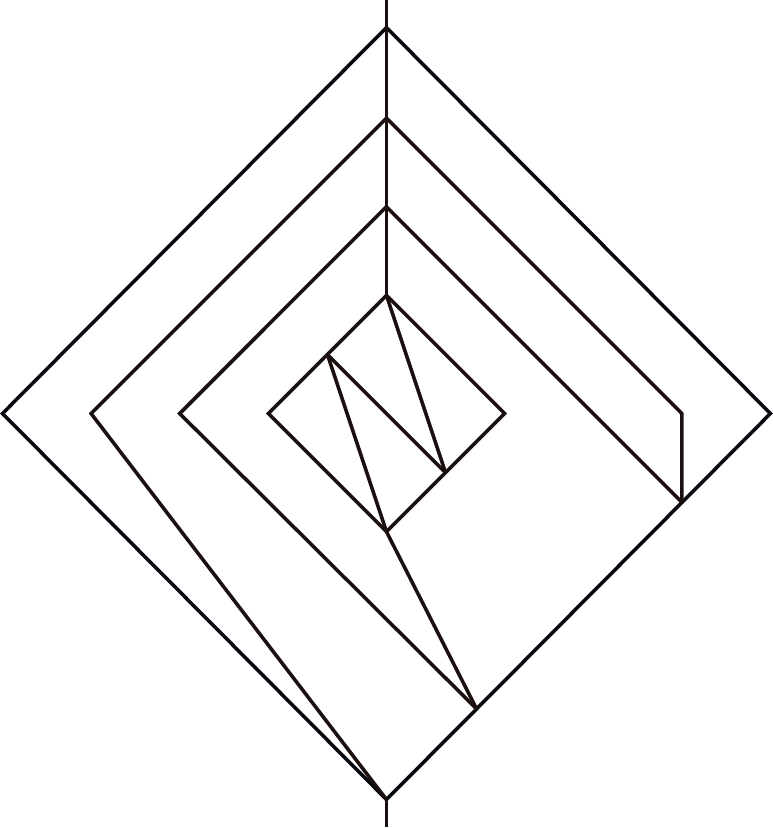}
    \caption{The element from \cref{fig: example diamond} has $l(f)=1^3$, $l(g)=1$ and $l(f\diamond g)=1^4$.}
    \label{fig:sum leaves}
\end{figure}

\begin{definition}
    \label{def: map phi alpha}
    Consider an address $\alpha$. Define $\varphi_{\alpha}$ as the map
    \[
        \varphi_{\alpha}\colon (F_3,\cdot) \to (F_3,\cdot)
    \]
    sending each $f\in F_3$ to the element $\varphi_{\alpha}(f)$, obtained by attaching the reduced pair of trees of $f$ in place of the address $\alpha$ of the upper tree of the pair of full ternary trees of depth $|\alpha|$, and then applying caret reduction.

    \begin{figure}[H]
    \centering
    \begin{tikzpicture}[scale=0.6]
    \node[anchor=south west,inner sep=0] at (0,0){\includegraphics[width=0.3\textwidth]{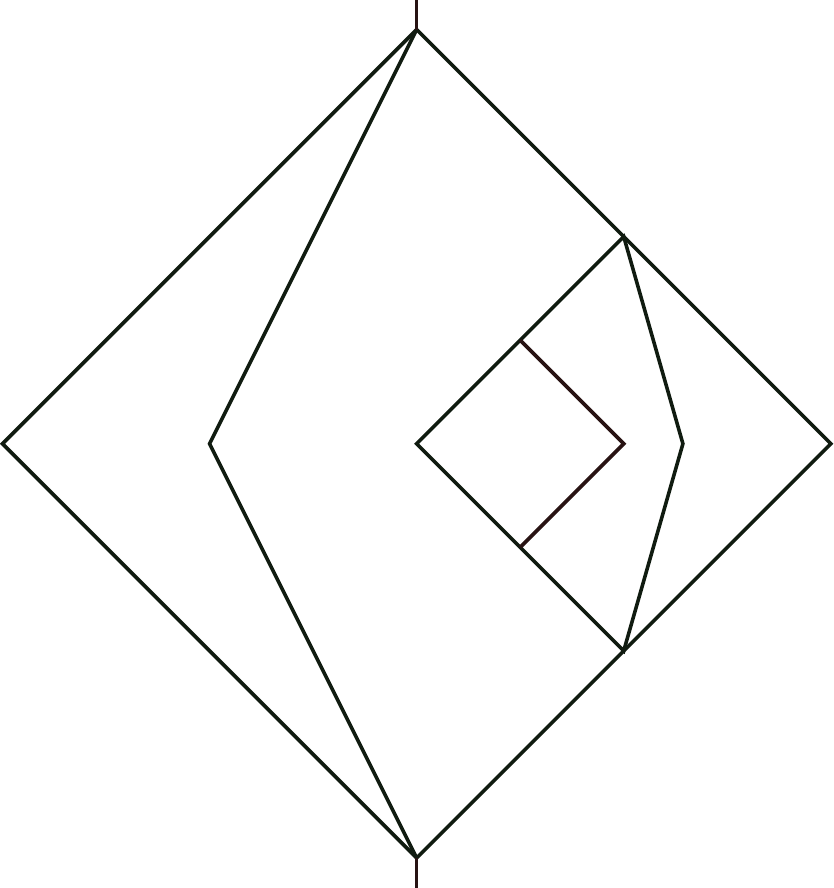}};
    \node at (3.9,3.4) {$f$};
    \end{tikzpicture}
    \caption{The element $\varphi_{20}(f)$.}
    \label{fig:example phi alpha}
    \end{figure}

    Let $\psi_{\alpha}$ be a linear isomorphism from $I_{\alpha}$ to $I$. Then, for any $f\in F_3$, its image under $\varphi_{\alpha}$ is the function 
    \begin{equation*}
        \varphi_{\alpha}(f)(t)= \left\{
        \begin{alignedat}{2}
            &t &&\quad t\in I\setminus I_{\alpha},\\
            &\psi_{\alpha}^{-1}f\psi_{\alpha}(t) &&\quad t\in I_{\alpha}.
        \end{alignedat}
        \right.
    \end{equation*}
    
    Observe that $\varphi_{\emptyset}=\id_{F_3}$.   
\end{definition}

In terms of intervals, the above example is represented by 
\begin{figure}[H]
    \centering
    \begin{tikzpicture}[scale=0.6]
    \node at (0,1) {$\varphi_{20}(f)=$};
    \node[anchor=south west,inner sep=0] at (1.5,0){\includegraphics[width=0.3\textwidth]{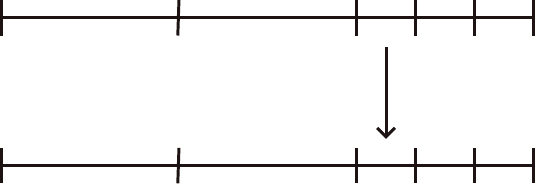}};
    \node at (5.5,1) {$f$};
    \node at (9,1) {.};
    \end{tikzpicture}
\end{figure}

The maps $\varphi_{\alpha}$ satisfy the following properties.

\begin{lemma}
\label{lemma: phi alpha is group hom}
    For any address $\alpha\in\{0,1,2\}^*$, the map $\varphi_{\alpha}$ is a group homomorphism.
\end{lemma}
\begin{proof}
    For any $\alpha\in\{0,1,2\}^*$, we have $\varphi_{\alpha}(1_{F_3})=1_{F_3}$.
    
    Consider now $f,g\in F_3$, then $\varphi_{\alpha}(fg)$ is the identity outside of $I_{\alpha}$ and, for all $t\in I_{\alpha}$, we have 
    \[
    \varphi_{\alpha}(fg)(t)=\psi_{\alpha}^{-1}(fg)\psi_{\alpha}(t)=(\psi_{\alpha}^{-1}f\psi_{\alpha})\cdot(\psi_{\alpha}^{-1}g\psi_{\alpha})(t)=\varphi_{\alpha}(f)\varphi_{\alpha}(g)(t).
    \]
    Therefore $\varphi_{\alpha}(fg)=\varphi_{\alpha}(f)\cdot\varphi_{\alpha}(g)$.   
\end{proof}

\begin{lemma}
\label{lemma:commuting phi}
    Let $\alpha,\beta\in\{0,1,2\}^*$ be such that $\alpha\npreceq\beta$ and $\beta\npreceq\alpha$. Then, for any $f,g\in F_3$, we have
    \[
        \varphi_{\alpha}(f)\cdot\varphi_{\beta}(g)=\varphi_{\beta}(g)\cdot\varphi_{\alpha}(f).
    \]
\end{lemma}
\begin{proof}
     Note that, if $\alpha\npreceq\beta$ and $\beta\npreceq\alpha$, then $I_{\alpha}\cap I_{\beta}=\emptyset$. Thus, the maps $\varphi_{\alpha}(f)$ and $\varphi_{\beta}(g)$ act as the identity on $I_{\beta}$ and $I_{\alpha}$ respectively, and the images commute.
\end{proof}

\begin{figure}[H]
        \centering
        \begin{tikzpicture}[scale=0.7]
        \node[anchor=south west,inner sep=0] at (0,0){\includegraphics[width=0.3\textwidth]{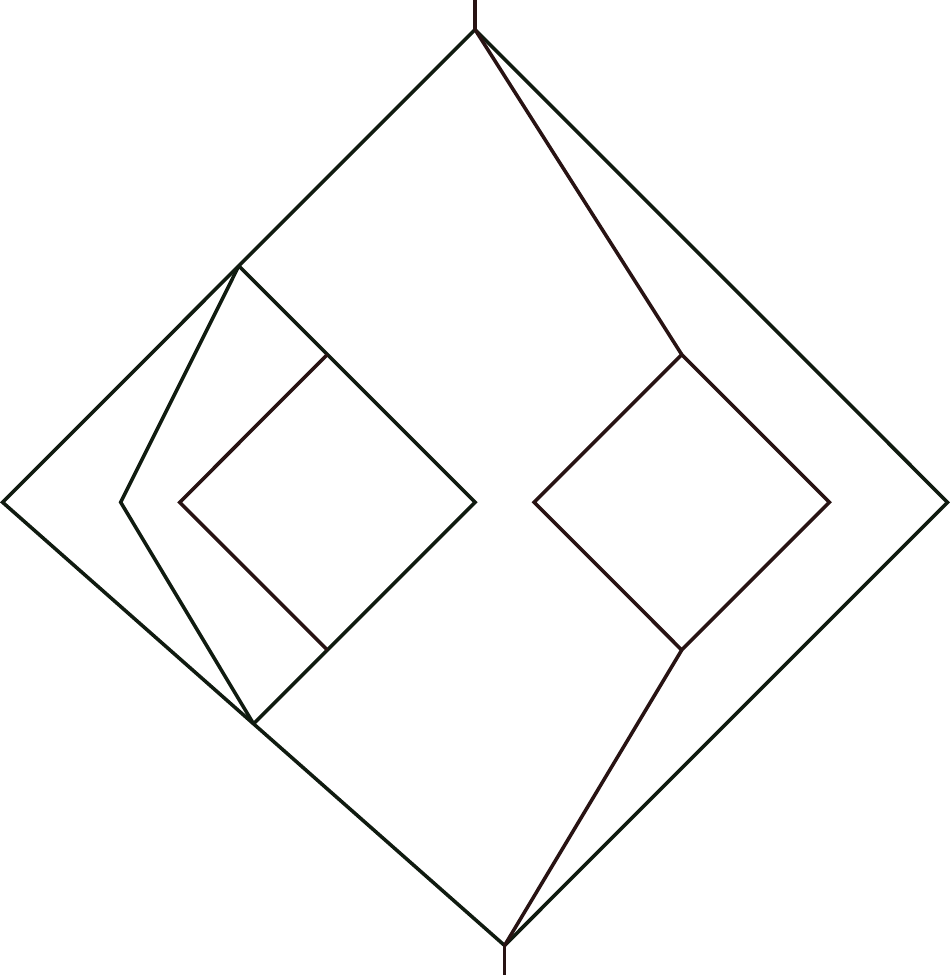}};
        \node at (1.8,2.8) {$f$};
        \node at (3.9,2.7) {$g$};
        \end{tikzpicture}
        \caption{The element $\varphi_{02}(f)\cdot\varphi_1(g)=\varphi_1(g)\cdot\varphi_{02}(f)$.}
        \label{fig:phi_alpha and phi_beta}
\end{figure}
    
We can now express $-\diamond -$ in terms of composition of functions.

\begin{lemma}
\label{lemma: new op and old op}
    Given $f,g \in F_3$ we have 
    \[
        f\diamond g= \varphi_{l(f)}(g)\cdot f= f\cdot \varphi_{f(l(f))}(g) .
    \]
\end{lemma}

\begin{proof}
    Consider $f,g\in F_3$, then $f\diamond g$ coincides with $f$ on $I \setminus I_{l(f)}$. On the interval $I_{l(f)}$ we have
    \[
    \begin{tikzcd}
    I_{l(f)} \arrow[r, "\simeq"] & I \arrow[r, "g"] & I \arrow[r, "\simeq"] & f(I_{l(f)}).
    \end{tikzcd}
    \]
    By definition of $\varphi_{l(f)}(g)$, the function $\varphi_{l(f)}(g)\cdot f$ acts as $f$ outside of~$I_{l(f)}$, and on $I_{l(f)}$ we have
    \[
    \begin{tikzcd}
    \varphi_{l(f)}(g)\cdot f\colon
    I_{l(f)} \arrow[r, "\psi_{l(f)}"] & I \arrow[r, "g"] & I \arrow[r, "\psi_{l(f)}^{-1}"] & I_{l(f)} \arrow[r, "f"] & f(I_{l(f)}).
    \end{tikzcd}
    \]
    But $l(f)$ is a leaf, therefore $f\colon I_{l(f)}\to f(I_{l(f)})$ is a linear isomorphism. Hence, we conclude that 
    \[
        f\diamond g= \varphi_{l(f)}(g)\cdot f.
    \]
    Similarly, one proves the second equality.

    Pictorially, the above argument correspond to the equality in \cref{fig:pictorial rep of fdiamg}.
\end{proof}

\begin{figure}[H]
        \centering
        \begin{tikzpicture}
        \node[anchor=south west,inner sep=0] at (0,0){\includegraphics[width=0.3\textwidth]{img/diamond_op/general_fdiamg.pdf}};
        \node at (1.9,2.2) {$g$};
        \node at (1.2,2.9) {$f$};
        \node at (4.2,2.1) {$=$};
        \node[anchor=south west,inner sep=0] at (4.7,0){\includegraphics[width=0.3\textwidth]{img/diamond_op/associativity.pdf}};
        \node at (6.65,2.1) {$g$};
        \node at (5.2,2.1) {$\cdots$};
        \node at (5.2,2.3) {$n$};
        \node at (9,2.1) {$\cdot$};
        \node[anchor=south west,inner sep=0] at (9.5,0){\includegraphics[width=0.3\textwidth]{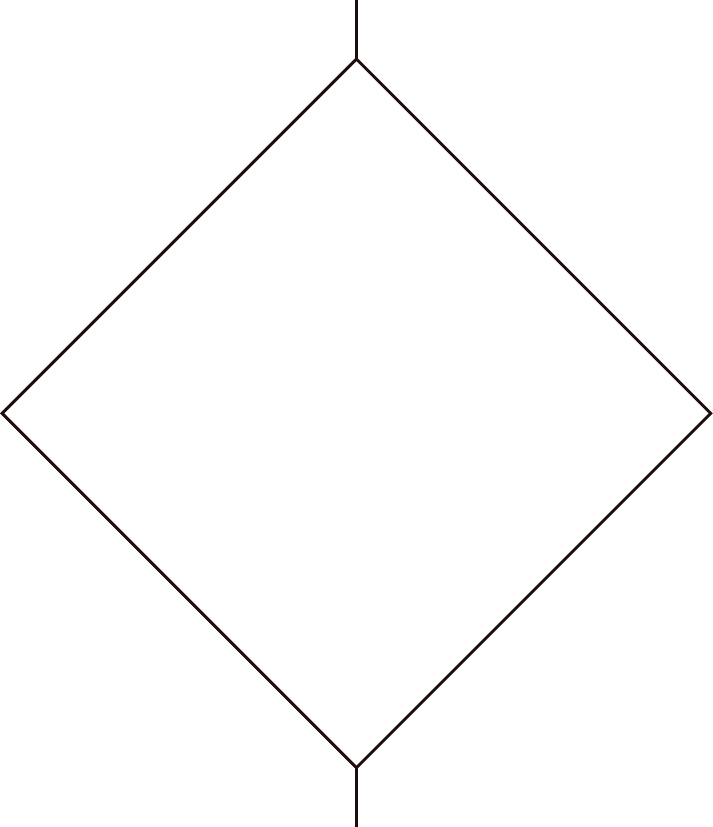}};
        \node at (11.35,2.2) {$f$};
        \end{tikzpicture}
        \caption{Pictorial representation of $f\diamond g=\varphi_{l(f)}(g)\cdot f$ for $l(f)=1^{n+2}$.}
        \label{fig:pictorial rep of fdiamg}
\end{figure}

With a similar reasoning, one obtains the following result.
\begin{lemma}
    For any $f,g\in F_3$ and $\alpha$ an address in $V_+(f)$ corresponding to a leaf of $f$, we have $f^{-1}\cdot\varphi_{\alpha}(g)\cdot f=\varphi_{f(\alpha)}(g)$.
\end{lemma}

Consider the function 
\[
    \Phi\colon (F_3, \diamond) \to (\Aut(F_3), \circ),
\] 
defined by $\Phi(f)= \varphi_{l(f)}$, the group homomorphism associated to the central leaf of $f$, for all $f\in F_3$.
\begin{lemma}
    The map $\Phi$ is a monoid homomorphism.
\end{lemma}
\begin{proof}
    By definition $\Phi(1)=\varphi_{\emptyset}=\id_{F_3}$. Consider $f,g\in F_3$, with central leaves $l(f)=1^n$ and $l(g)=1^m$, then
    \[
        \Phi(f\diamond g)= \varphi_{l(f\diamond g)} =\varphi_{1^{n+m}}=\varphi_{1^n}\circ \varphi_{1^m}=\varphi_{l(f)} \circ\varphi_{l(g)}=\Phi(f)\circ\Phi(g).
    \]
    Hence, $\Phi$ is a monoid homomorphism.
\end{proof}

The map $\Phi$ gives an action of the monoid $(F_3, \diamond)$ on $(F_3, \cdot)$. Note also that $\Phi(f\diamond g)=\Phi(g\diamond f)$, since the central leaves $l(f\diamond g)$ and $l(g\diamond f)$ have the same address.

The maps defined satisfy the following relation.
\begin{proposition}
    For all $f,g,h\in F_3$ we have
    \[
    \varphi_{l(f)}(g\diamond h)=\varphi_{l(f\diamond g)}(h)\cdot \varphi_{l(f)}(g).
    \]
   
\end{proposition}
\begin{proof}
    Using the fact that $\Phi$ is a monoid homomorphism together with \cref{lemma: new op and old op}, we have
    \begin{align*}
        \varphi_{l(f)}(g\diamond h)
        &= \varphi_{l(f)}(\varphi_{l(g)}(h)\cdot g)= \\
        &= \varphi_{l(f)}(\varphi_{l(g)}(h))\cdot\varphi_{l(f)}(g)=\\
        &= \varphi_{l(f\diamond g)}(h)\cdot\varphi_{l(f)}(g).
        \qedhere
    \end{align*}
\end{proof}

\subsection{Other monoids}
\label{subsec: other monoids}
In the definition of $(F_3, \diamond)$ we fixed the central leaf as our address for reasons coming from low-dimensional topology (see \cref{sec: connected sum}), but one could define similar monoids using a different address.

In order to obtain an operation on $F_3$ analogous to $\diamond$, we need to have a well-defined choice of address. Fix $f\in F_3$ and $\alpha\in V_+(f)$ corresponding to a leaf of $f$. Consider the map
\[
    f\diamond_{\alpha}-\colon F_3 \to F_3,
\]
sending $g\in F_3$ to $f\diamond_{\alpha}g$, obtained by attaching $g$ at the address $\alpha$ of $f$. Note that $f\diamond_{\alpha}g=\varphi_{\alpha}(g)\cdot f$.

\begin{definition}
    For $i\in\{0,1,2\}$, define the operation $-\diamond_i -$ on $F_3$ as
    \[
        f\diamond_i g:=f\diamond_{\alpha_i}g,
    \]
    where $\alpha_i$ is the address corresponding to the leaf $i^m$ of $f$. 
\end{definition}
Note that for $f=1_{F_3}$ we have $i^m=\emptyset$, the root of the upper tree, hence $1_{F_3}\diamond_i g=g\diamond_i 1_{F_3}=g$ for all $g\in F_3$. Moreover $\diamond_1$ gives back $\diamond$.

\begin{proposition}
    The pair $(F_3, \diamond_i)$ is a monoid for all $i\in\{0,1,2\}$.
\end{proposition}
\begin{proof}
    Similarly to \cref{prop: diamond gives a monoid}, one observes that $1_{F_3}$ is the identity element and, for all $f,g,h\in F_3$, we have $(f\diamond_i g)\diamond_i h=f\diamond_i (g\diamond_i h)$.
\end{proof}

As a consequence of \cref{lemma:commuting phi} we have the following corollary.
\begin{corollary}
    Given $i,j\in \{0,1,2\}$ such that $i\neq j$, and $f,g,h\in F_3$, we have
    \[
        (f\diamond_i g)\diamond_j h= (f\diamond_j h)\diamond_i g.
    \]
\end{corollary}

\begin{figure}[H]
\centering
\begin{tikzpicture}
\node[anchor=south west,inner sep=0] at (0,0){\includegraphics[width=0.35\textwidth]{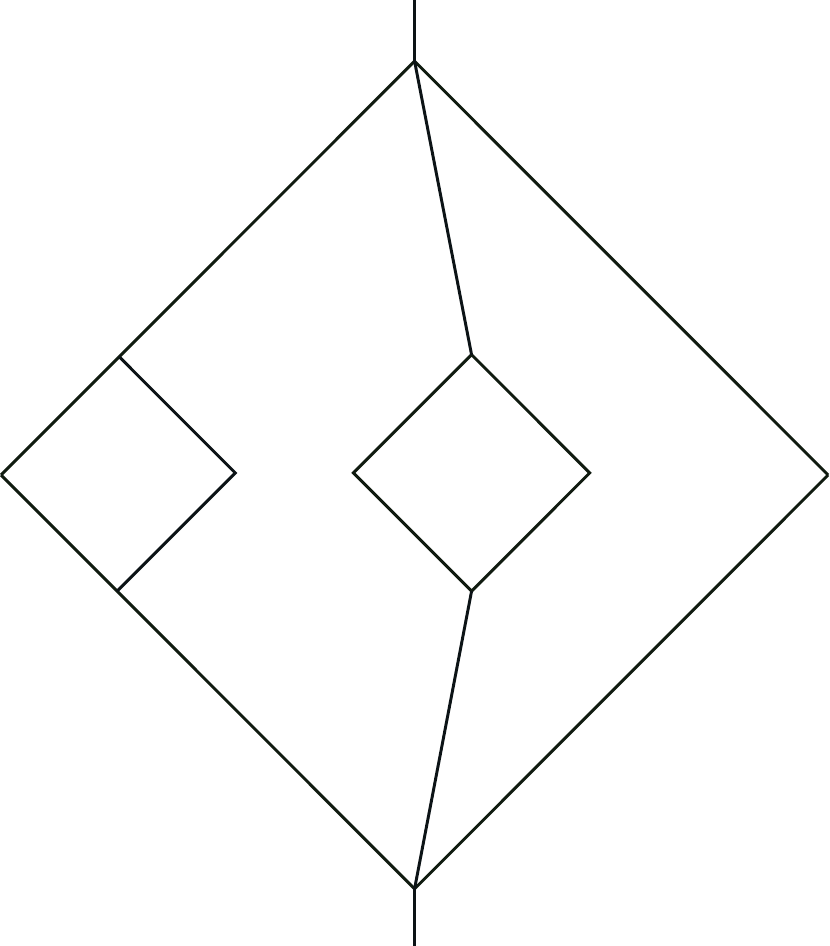}};
\node at (0.6,2.5) {$g$};
\node at (2.5,2.6) {$h$};
\node at (1.6,3.3) {$f$};
\end{tikzpicture}
\caption{The element $(f\diamond_0 g)\diamond_1 h= (f\diamond_1 h)\diamond_0 g$.}
\label{figure:diamond i j}
\end{figure}

%% file: connected_sum.tex
\section{Connected sum and prime decomposition}
\label{sec: connected sum}

Using Jones' map $\cL\colon F_3 \to \mathsf{Links}$, and its extension $\cL_*\colon F_3 \to \mathsf{Links}_*$, we translate the results of \cref{sec: central monoid} into low-dimensional topology. Specifically, we focus on the connection with connected sum and prime decomposition.

Note that, for any $f\in F_3$, the central component of $\cL_*(f)$ contains the central leaf $l(f)$ by construction. Consider the following modified version of the map $\cL$.
\begin{definition}
    Define $\cK\colon F_3 \to \mathsf{Knots}$ to be the map associating to each element $f\in F_3$ the knot $\cK(f)$, given by the central component of~$\cL_*(f)$.
\end{definition}

Note that, if $\cL(f)$ is connected, then $\cK(f)=\cL(f)$. In particular, by \cref{thm:Alexander type thm}, this implies that the map $\cK$ is surjective.

Given $\cK(f)$ we consider the orientation on it obtained by closing the central component to the left of the diagram and orienting it clockwise (see \cref{fig:oriented central component}).

\begin{figure}[H]
    \centering
    \includegraphics[width=0.3\linewidth]{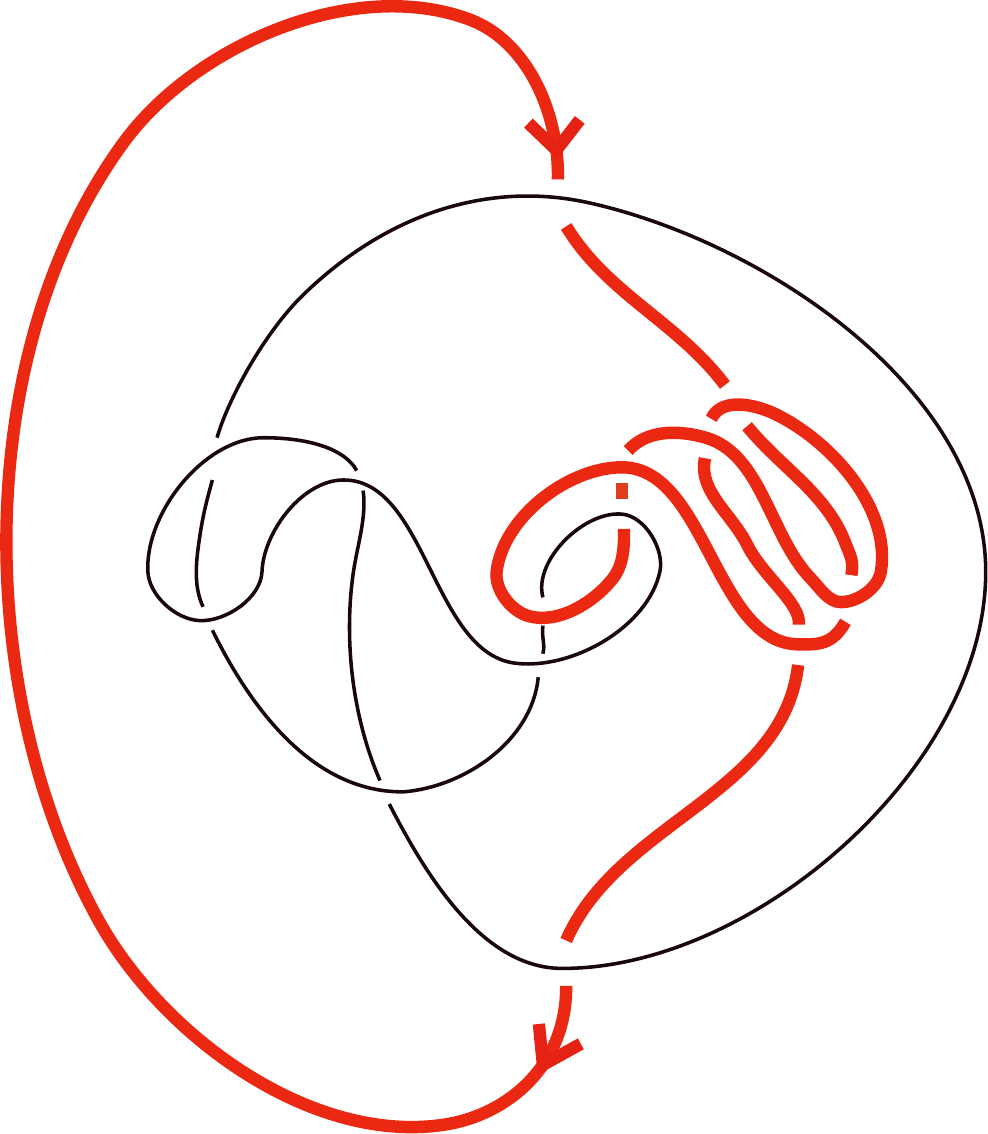}
    \caption{Example of $\cK(f)$ (bold) with assigned orientation.}
    \label{fig:oriented central component}
\end{figure}

Note that, if $f\in F_3$ has image $\cK(f)$ with orientation induced by the above convention, the element $\Bar{f}\in F_3$ obtained by rotating $f$ by $180$ degrees has image $\cK(f)$ with opposite orientation.

The $-\diamond -$ operation corresponds to connected sum.

\begin{proposition}
\label{prop: connected sum of knots}
    For all $f,g\in F_3$ we have 
    \[
        \cK(f\diamond g)=\cK(f)\#\cK(g) .
    \]
    Moreover $\cK(f\diamond g)=\cK(g\diamond f)$.
\end{proposition}
\begin{proof}
    Observe that the choice of component in the definition of $\cK$ guarantees that the orientation convention is respected by the $-\diamond-$ operation, giving us a well-defined connected sum of knots -- see example \cref{fig:example conn sum}.
    \begin{figure}[H]
    \centering
    \includegraphics[width=0.35\linewidth]{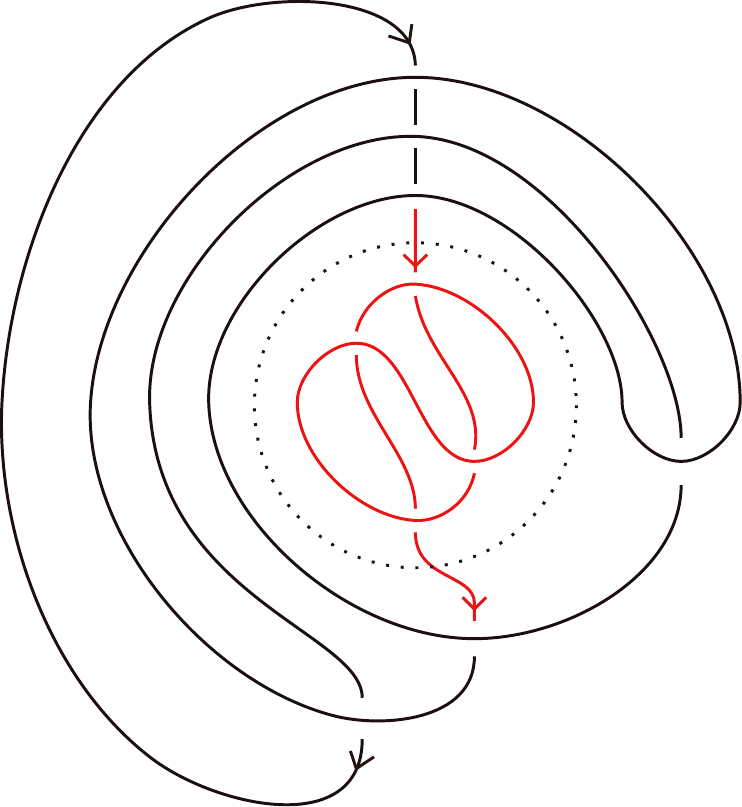}
    \caption{Example of connected sum from the elements in \cref{fig: example diamond}.}
    \label{fig:example conn sum}
    \end{figure}
    
    The connected sum of knots is commutative, therefore
    \[
        \cK(f\diamond g)=\cK(f)\#\cK(g)=\cK(g)\#\cK(f)=\cK(g\diamond f) .\qedhere
    \]
\end{proof}

Applying \cref{prop: connected sum of knots}, we obtain the following result.
\begin{theorem}
    The map $\cK$ gives a surjective homomorphism of monoids from the central monoid to knots with connected sum
    \[
        (F_3, \diamond) \twoheadrightarrow (\mathsf{Knots}, \#). 
    \]
\end{theorem}

For $f,g \in F_3$ their image under $\cL$ could result in links with more than one component, for which the connected sum is not well-defined. Consider the monoid of pointed links $(\mathsf{Links_*}, \#_*)$, with operation given by connected sum on the distinguished components.  With this notion, we can extend the above result by taking the connected sum of Thompson links to be the connected sum as pointed links.
\begin{corollary}
\label{cor: connected sum of links}
    For all $f,g\in F_3$ we have 
    \[
        \cL_*(f\diamond g)=\cL_*(f)\#_*\cL_*(g).
    \]
    In particular
    \[
        \cL_*(f\diamond g)=\cL_*(g\diamond f) .
    \]
\end{corollary}

\begin{remark}
    In order to consider connected sum we had to orient the central component. Hence, when looking at monoids with connected sum, we are working with oriented knots or pointed links with oriented marked component. The canonical choice of orientation made above ensures that the connected sum on the central leaf does not require us to invert orientations when taking the image of $-\diamond-$ under~$\cL_*$.
\end{remark}

Denote by $U\colon (\mathsf{Links_*}, \#_*)\to (\mathsf{Knots}, \#)$ the morphism obtained by sending each pointed link to its marked component. Then, applying the above results, we have the following theorem.

\begin{restate}{Theorem}{thm:monoid diagram}
    The following is a commutative diagram of surjective monoid homomorphisms.
    \begin{center}
    \begin{tikzcd}
    (F_3, \diamond) \arrow[rr, "\cL_*", two heads] \arrow[rrdd, "\cK"', two heads] &  & (\mathsf{Links_*}, \#_*) \arrow[dd, "U", two heads] \\
                                     &  &                   \\
                                     &  & (\mathsf{Knots}, \#)                
    \end{tikzcd}    
    \end{center}
\end{restate}

\subsection{Prime decomposition}
Given any knot $K$, we can now describe, up to prime components, how to obtain an element $f$ of $F_3$ in a standard form such that $\cL(f)=K$.

Recall that a knot is said to be prime if it cannot be expressed as the connected sum of two non-trivial knots. Any knot $K$ has a prime decomposition $K=P_1\#\dots\#P_n$, i.e.~it can be written as the connected sum of prime knots $P_1, \dots, P_n$, unique up to permutation \cite[Theorem 2.12]{Lickorish97}.

\begin{theorem}
\label{thm:prime decomposition}
    Let $K$ be a knot with prime decomposition $K=P_1\#\dots\#P_n$. There exists an element $f\in F_3$ of the form 
    \[
    f=f_1\diamond\dots\diamond f_n
    \]
    such that $\cL(f)=K$ and $\cL(f_i)=P_i$ for all $i=1, \dots, n$.
\end{theorem}
\begin{proof}
    By \cref{thm:Alexander type thm}, for all knots $P_i$ there exists an element $f_i\in F_3$ such that $\cL(f_i)=P_i$. Fix $f=f_1\diamond\dots\diamond f_n$, then, by \cref{cor: connected sum of links}, we have 
    \[
        \cL(f)=U_*\cL_*(f_1\diamond\dots\diamond f_n)=P_1\#\dots\#P_n=K. \qedhere
    \]
\end{proof}

\begin{remark}
    Note that, using commutativity of connected sum, this gives $n!$ elements with the same image under $\cL$, with $n$ the number of components in the prime decomposition. Not all elements $f\in F_3$ with image $K$ will be of the form $f=f_1\diamond\dots\diamond f_n$, but \cref{thm:prime decomposition} guarantees that we will always find such a representative in $F_3$.
\end{remark}

Let $U(F_3)$ denote the set of elements of $F_3$ corresponding to the unknot.
\begin{corollary}
    A knot $K$ is prime if and only if, for all $f\in F_3$ with image $\cL(f)= K$, if $f=f_1\diamond f_2$ then at least one of $f_1, f_2$ is in $U(F_3)$.
\end{corollary}

In analogy to the theory linking braids and knots, one would like to have a Markov type theorem for Thompson's group, i.e.~a list of moves such that $\cL(f)=\cL(g)$ if and only if $f$ and $g$ are related by a finite sequence of these moves. The above results give us possible moves for a Markov type theorem on $F_3$. In particular, we have
\[
    \cL(f\diamond g)=\cL(g\diamond f).
\]

\begin{figure}[H]
    \centering
        \begin{tikzpicture}
        \node[anchor=south west,inner sep=0] at (0,0){\includegraphics[width=0.25\textwidth]{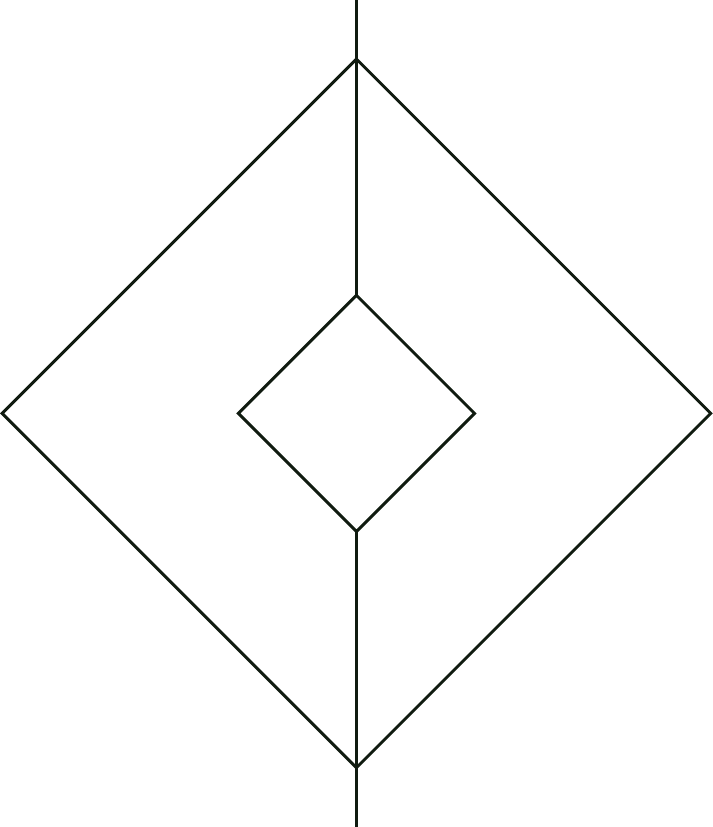}};
        \node at (1.6,1.8) {$g$};
        \node at (1,2.3) {$f$};
        \node at (4,1.8) {$\sim$};
        \node[anchor=south west,inner sep=0] at (5,0){\includegraphics[width=0.25\textwidth]{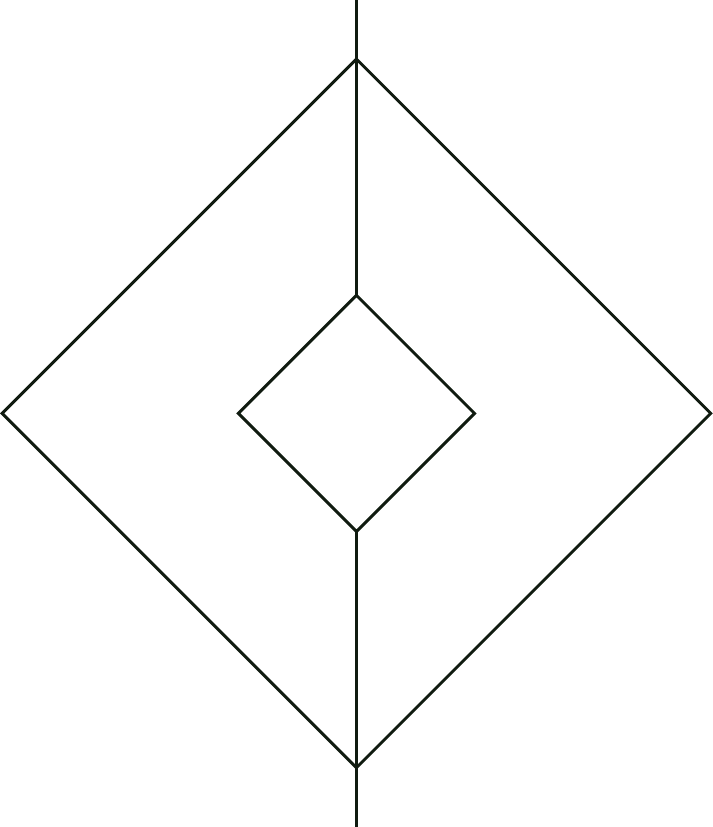}};
        \node at (6.6,1.8) {$f$};
        \node at (6,2.3) {$g$};
        \end{tikzpicture}
\end{figure}

And, for all $u\in F_3$ such that $\cL(u)$ is the unknot, we have
\[
    \cL(f\diamond u)=\cL(u\diamond f)=\cL(f).
\]
\begin{figure}[H]
    \centering
        \begin{tikzpicture}
        \node[anchor=south west,inner sep=0] at (0,0){\includegraphics[width=0.25\textwidth]{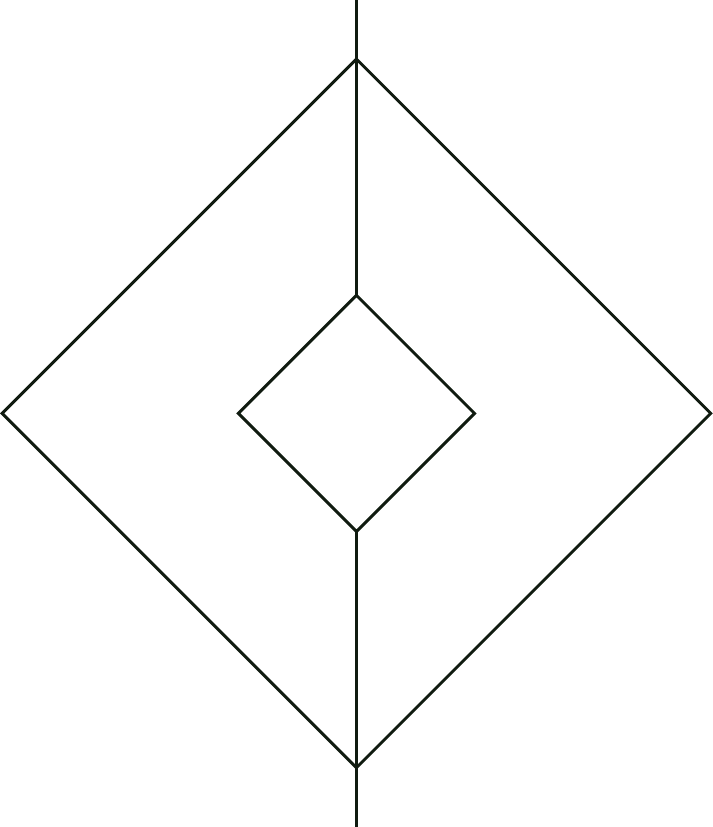}};
        \node at (1.6,1.9) {$u$};
        \node at (0.9,2.3) {$f$};
        \node at (4.1,1.9) {$\sim$};
        \node[anchor=south west,inner sep=0] at (5,0){\includegraphics[width=0.25\textwidth]{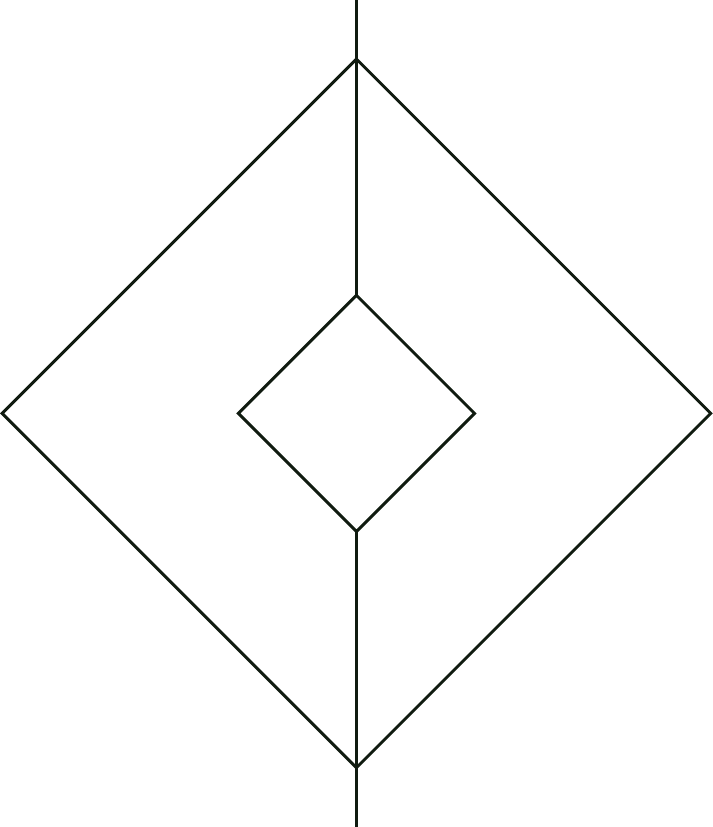}};
        \node at (6.6,1.9) {$f$};
        \node at (5.9,2.3) {$u$};
        \node at (9.1,1.9) {$\sim$};
        \node[anchor=south west,inner sep=0] at (10,0){\includegraphics[width=0.25\textwidth]{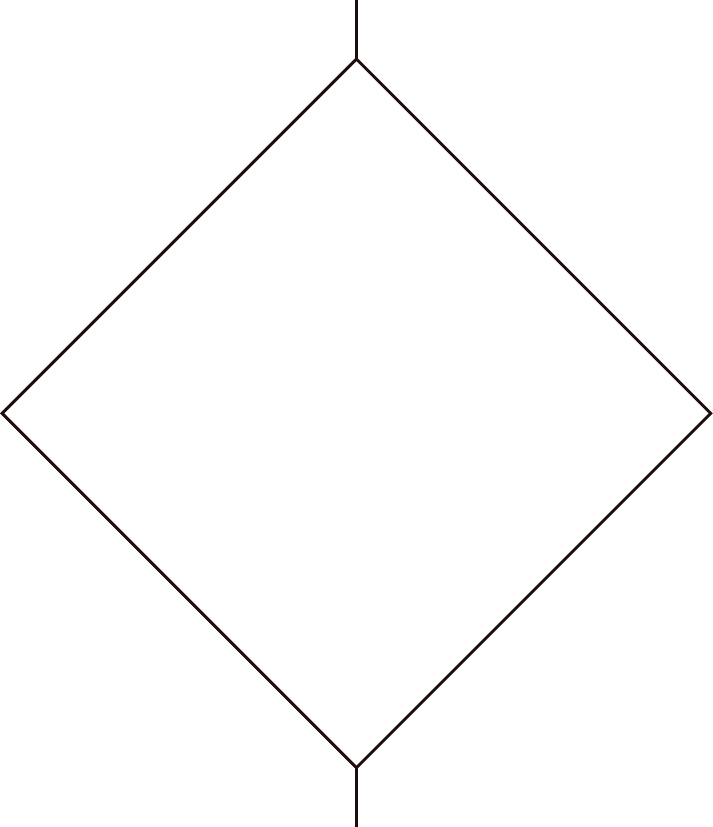}};
        \node at (11.5,1.9) {$f$};
        \end{tikzpicture}
\end{figure}

These moves are not sufficient. For example, the following two elements of $F_3$ both have image the unknot, but they are not related by the moves above. 

\begin{figure}[H]
    \centering
        \begin{tikzpicture}
        \node[anchor=south west,inner sep=0] at (0,0){\includegraphics[width=0.2\textwidth]{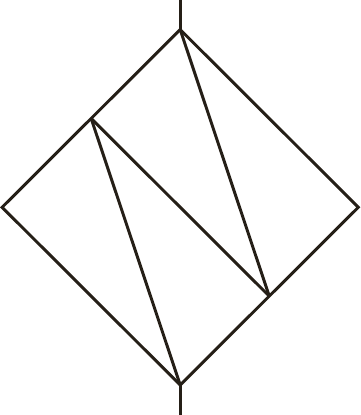}};
        \node at (3.25,1.35) {$\sim$};
        \node[anchor=south west,inner sep=0] at (4,0){\includegraphics[width=0.2\textwidth]{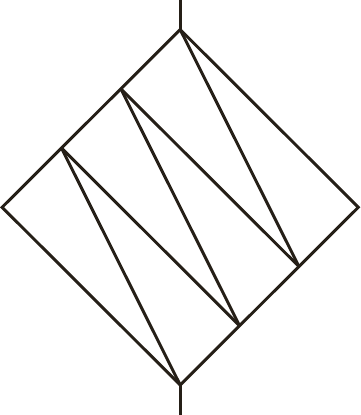}};
        \end{tikzpicture}
\end{figure}

\begin{remark}
\label{rmk:connected sum on alfa}
    As observed in \cref{subsec: other monoids}, given $f\in F_3$ and an address $\alpha\in V_+(f)$ corresponding to a leaf, we can consider the operation $f\diamond_{\alpha}-$. Similarly to $-\diamond-$, the map $f\diamond_{\alpha}-$ also corresponds to a connected sum of components. In this case, the components involved in the connected sum strongly depend on the address $\alpha$ considered. Specifically, the link $\cL(f\diamond_{\alpha}g)$ is obtained by doing connected sum between the component of $\cL(f)$ containing the image of $\alpha$ and the central component $\cK(g)$.

    Note that a Markov type theorem should incorporate when different addresses correspond to the same connected sum. This can be done via an algorithm similar to that used to define Thompson permutations by Aiello and Iovieno in \cite{aiello2022computationalstudynumberconnected}, identifying when two leaves belong to the same component of $\cL(f)$, and then checking the orientation.
\end{remark}

%% file: link_rep.tex
\section{Constructing link representatives}
\label{sec:link representatives}
We now aim to extend the standard form of \cref{sec: connected sum}. In particular, we build standard representatives in $F_3$ for certain links. To do this, we define operations on $F_3$ which correspond to operations on links, much like $-\diamond -$ and connected sum.

\subsection{Disjoint union}
The first step to pass from knots to links is to define an operation on $F_3$ corresponding to disjoint union. 

\begin{proposition}
    Given $f,g\in F_3\setminus\{1_{F_3}\}$ define
    \[
        \sqcup(f,g):=\varphi_1(f)\varphi_0(g)=\varphi_0(g)\varphi_1(f).
    \]
    Then $\cL(\sqcup(f,g))$ is the disjoint union of $\cL(f)$ and $\cL(g)$.
\end{proposition}

\begin{proof}
    Consider $f,g\in F_3$, then $\sqcup(f,g)$ corresponds to the pair of trees represented below and, applying Jones' algorithm, we obtain the link on the right, i.e.~the disjoint union.
    \begin{figure}[H]
        \centering
        \begin{tikzpicture}
        \node[anchor=south west,inner sep=0] at (0,0.5){\includegraphics[width=0.35\linewidth]{img/diamond_op/diamond_ij.pdf}};
        \node at (0.6,3) {$g$};
        \node at (2.5,3) {$f$};
        \node at (5.5,3) {$\longrightarrow$};
        \node at (5.5,3.4) {$\cL$};
        \node[anchor=south west,inner sep=0] at (7,0){\includegraphics[width=0.35\linewidth]{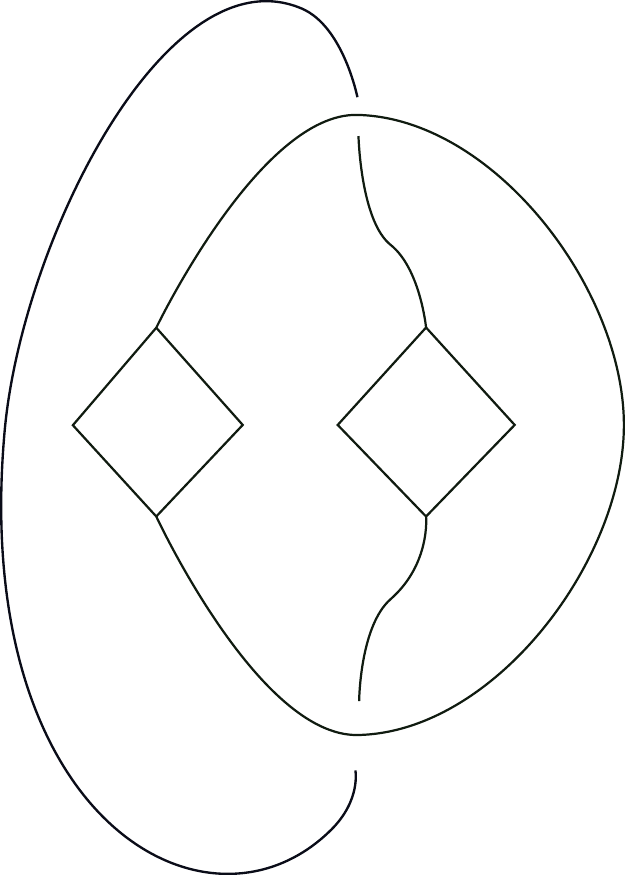}};
        \node at (8.1,3.2) {$\cL(g)$};
        \node at (10.05,3.2) {$\cL(f)$};
        \end{tikzpicture}
        \label{fig:disjoint union}
        \qedhere
    \end{figure}
\end{proof}

Note that we are taking $f,g\neq 1_{F_3}$ to avoid possible problems related to the fact that we are considering reduced pairs of trees. Hence, to obtain the unlink, one should take representatives in $U(F_3)\setminus\{1_{F_3}\}$, for example $y_0$.

The operation $\sqcup(-,-)$ behaves well with respect to the product in~$F_3$.
\begin{lemma}
    The map $\sqcup\colon F_3\times F_3 \to F_3$ satisfies
   \[
        \sqcup(f\cdot f', g\cdot g')=\sqcup(f,g)\cdot\sqcup(f',g').
   \] 
\end{lemma}

\begin{proof}
    Consider $f,f',g,g'\in F_3$. By \cref{lemma:commuting phi}, the images of the maps~$\varphi_0$ and $\varphi_1$ commute, so we have
    \begin{align*}
        \sqcup(f\cdot f', g\cdot g')
        &= \varphi_1(ff')\varphi_0(gg')=\\
        &= \varphi_1(f)\varphi_1(f')\varphi_0(g)\varphi_0(g')=\\
        &= \varphi_1(f)\varphi_0(g)\varphi_1(f')\varphi_0(g')=\\
        &= \sqcup(f,g)\cdot\sqcup(f',g').
        \qedhere
    \end{align*}
\end{proof}

Note that
\[
    \cL(\sqcup(f,g))=\cL(\sqcup(g,f)),
\]
and one could substitute $\varphi_0$ with $\varphi_2$ in the definition of $\sqcup$, obtaining a total of $4$ elements in $F_3$ with image the disjoint union of $\cL(f)$ and~$\cL(g)$. In general, given $n$ links $L_1, \dots, L_n$, and $f_1, \dots, f_n$ representatives in $F_3\setminus\{1_{F_3}\}$, we associate to $L_1\sqcup\dots\sqcup L_n$ the element
\[
    \sqcup(f_1,\dots,f_n):=\sqcup(\sqcup(\dots(\sqcup(f_1,f_2)\dots),f_n).
\]

\begin{remark}
    Permuting the terms in the above operation, and using both $\varphi_0$ and $\varphi_2$, gives $c(n)$ elements with image $L_1\sqcup\dots\sqcup L_n$, where~$c(1)=1$ and $c(n)$ is given by the recursive formula

    \[
    c(n)=
    \begin{cases}
        4\sum_{i=1}^{n/2 -1} \binom{n}{i}c(i)c(n-i)+2\binom{n}{n/2}c(n/2)^2 & n \text{ even,} \\
        4\sum_{i=1}^{\lfloor n/2 \rfloor} \binom{n}{i}c(i)c(n-i) & n \text{ odd.} \\
    \end{cases} 
    \]

    This can be rewritten using Catalan numbers $\mathrm{Cat(n)}$ as 
    \[
        c(n)=2^{n-1}n!\mathrm{Cat}(n-1).
    \]
\end{remark}

\subsection{Linking components}
The aim of this section is to define local moves in $F_3$ that correspond to linking the central components of two given Thompson links, and use them to construct representatives. To do so we consider the elements $H^{\pm}$ in $F_3$ corresponding to the pairs of trees in \cref{fig:Hpm in F3}.

\begin{figure}[H]
    \centering
    \begin{tikzpicture}
        \node[anchor=south west,inner sep=0] at (0,0){\includegraphics[width=0.25\textwidth]{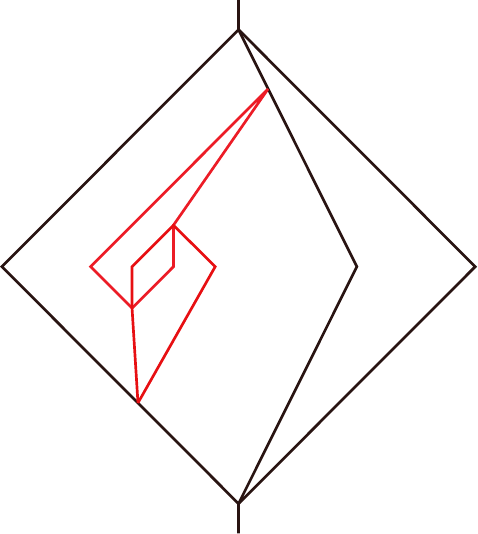}};
        \node[anchor=south west,inner sep=0] at (5,0){\includegraphics[width=0.25\textwidth]{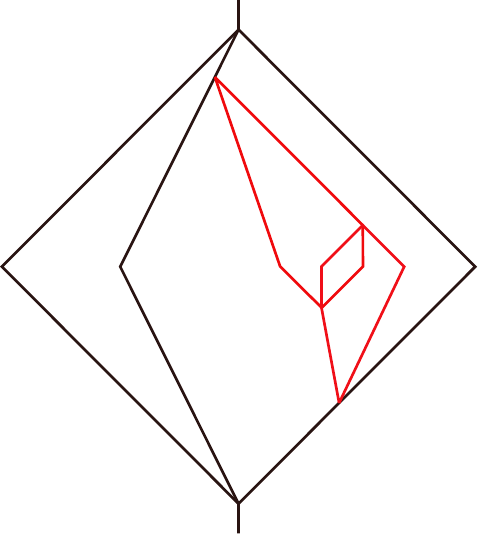}};
    \end{tikzpicture}
    \caption{The elements $H^+$ (left) and $H^-$ (right).}
    \label{fig:Hpm in F3}
\end{figure}

Orient the links obtained so that the central components respect the convention of \cref{sec: connected sum} and the images $\cL(H^{\pm})$ are, respectively, the positive and negative Hopf link.

Define the \textit{linking moves} $l_{\pm}\colon F_3\times F_3\to F_3$ as 
\[
    l_{+}(f,g):=(H^{+}\diamond_{12}f)\diamond_0 g
\]
and
\[
    l_{-}(f,g):=(H^{-}\diamond_{10}f)\diamond_0 g.
\]
Note that $H^{\pm}=l_{\pm}(1_{F_3},1_{F_3})$.

\begin{figure}[H]
    \centering
    \begin{tikzpicture}
        \node[anchor=south west,inner sep=0] at (0,0){\includegraphics[width=0.33\textwidth]{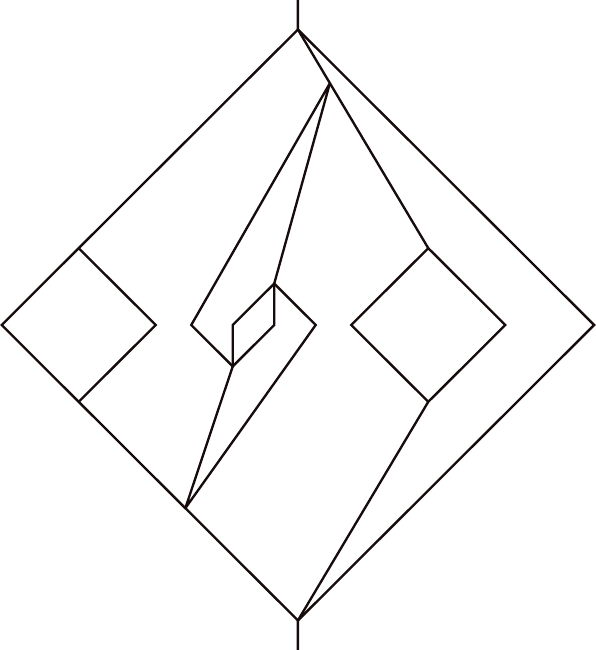}};
        \node at (0.58,2.3) {$g$};
        \node at (3,2.3) {$f$};
    \end{tikzpicture}
    \caption{The element $l_+(f,g)$.}
    \label{fig:bridge move 1}
\end{figure}

The orientation fixed on $\cL(H^{\pm})$ is such that the central components of $f$ and $g$ respect the convention of \cref{sec: connected sum}.

Starting from $H^1=H^+$, define $H^n\in F_3$ as the element obtained from $H^{n-1}$ by adding vertices in place of the two leaves of the upper tree with addresses $11(2)^{n-2}0$ and $11(2)^{n-1}$, and the tree shown below at the leaf $0(2)^{n-1}$ of the lower tree.
\[
    \includegraphics[width=0.15\linewidth]{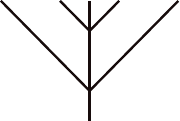}
\]

\begin{example}
For $n=2$ and $n=3$ we have 
\[
\begin{tikzpicture}
        \node[anchor=south west,inner sep=0] at (1,0){\includegraphics[width=0.25\textwidth]{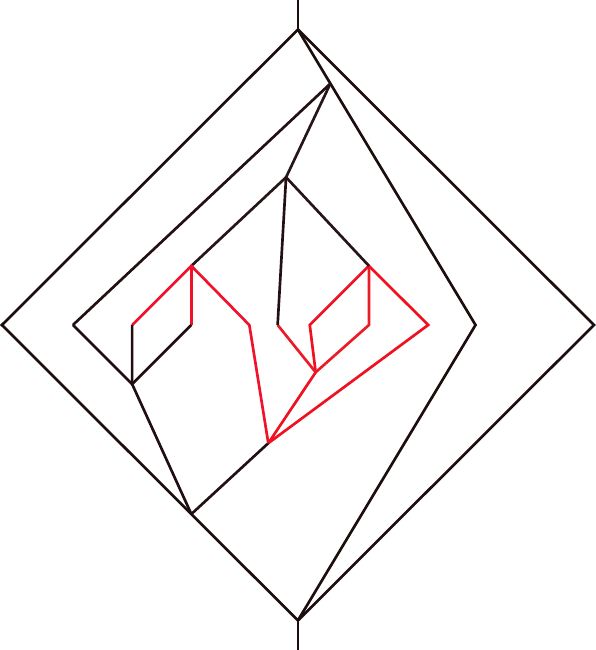}};
        \node at (0,1.7) {$H^2=$};
        \node at (5,1.7) {and};
        \node at (6,1.7) {$H^3=$};
        \node[anchor=south west,inner sep=0] at (7,0){\includegraphics[width=0.25\textwidth]{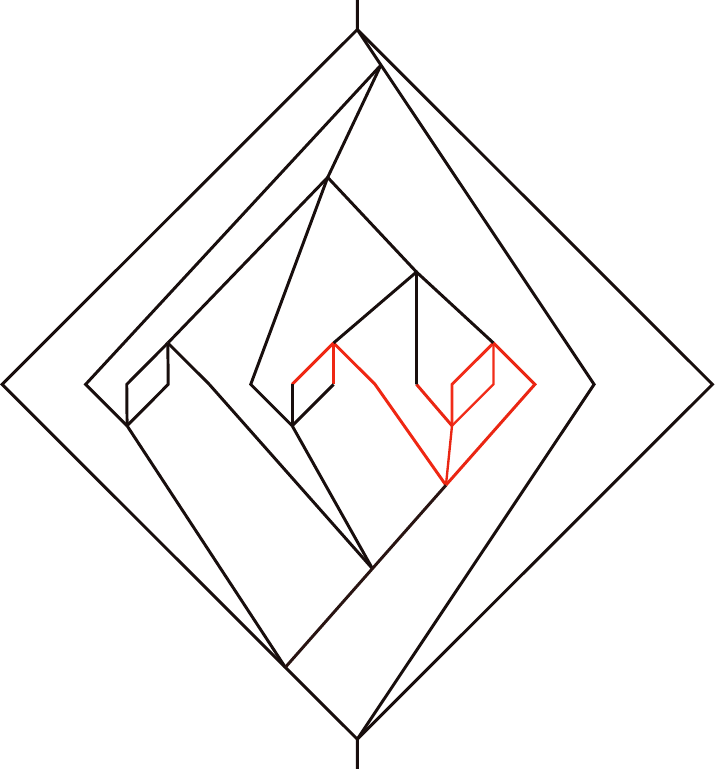}};
        \node at (10.5,1.7) {.};
\end{tikzpicture} 
\]
\end{example}

Similarly, define $H^{-n}$ recursively from $H^{-1}=H^-$. The element $H^{-n}$ is obtained from $H^{-(n-1)}$ by adding vertices in place of the leaves $2(0)^{n}$ and $2(0)^{n-1}2$ in the lower tree, and the tree
\[
    \includegraphics[width=0.15\linewidth]{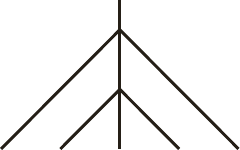}
\]
at the address $11(0)^{n-2}$ of the upper tree.
\[
\begin{tikzpicture}
        \node[anchor=south west,inner sep=0] at (1,0){\includegraphics[width=0.25\textwidth]{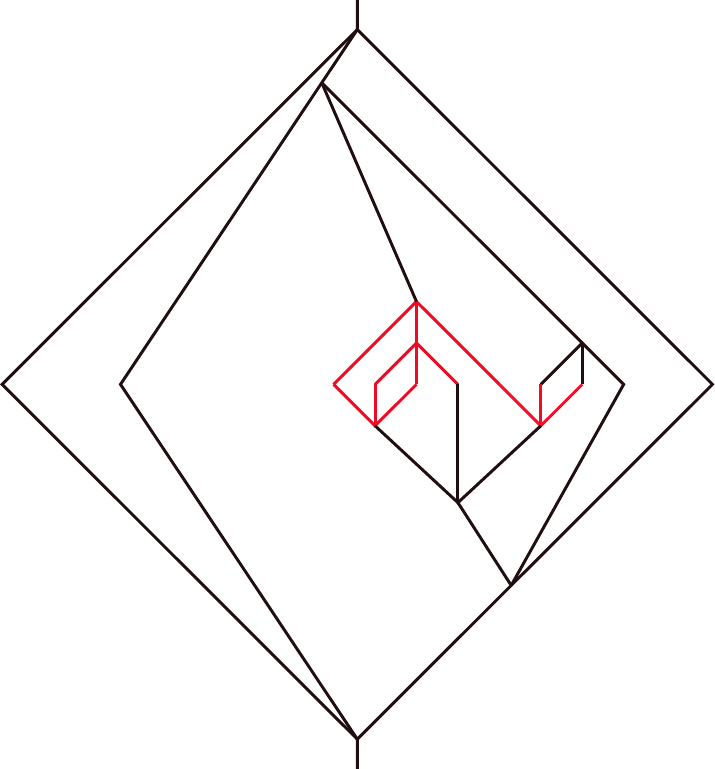}};
        \node at (0,1.7) {$H^{-2}=$};
        \node at (5,1.7) {and};
        \node at (6.1,1.7) {$H^{-3}=$};
        \node[anchor=south west,inner sep=0] at (7,0){\includegraphics[width=0.25\textwidth]{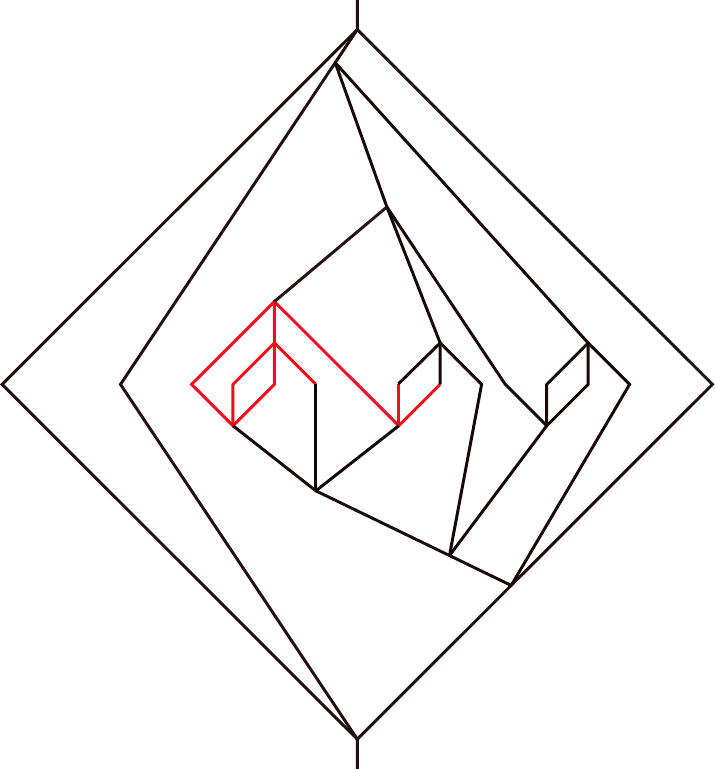}};
        \node at (10.5,1.7) {.};
\end{tikzpicture}
\]

We can now define a generalised version of the linking moves.
\begin{definition}
    For any $n\in\N$ define the maps $l_{\pm n}\colon F_3\times F_3 \to F_3$ by
    \[
        l_n(f,g):=(H^n\diamond_{12} f) \diamond_0 g \;\;\text{ and }\;\;
        l_{-n}(f,g):=(H^{-n}\diamond_{10} f) \diamond_0 g.
    \]
\end{definition}

The links $\cL (l_{*}(f,g))$ and $\cL (l_{-*}(f,g))$ have local diagrams as in \cref{fig:positive bridge link} and \cref{fig:negative bridge link} respectively, linking the central components $\cK(f)$ and~$\cK(g)$ of $\cL(f)$ and $\cL(g)$.

\begin{figure}[H]
        \centering
        \begin{tikzpicture}
        \node[anchor=south west,inner sep=0] at (0,0){\includegraphics[width=0.5\textwidth]{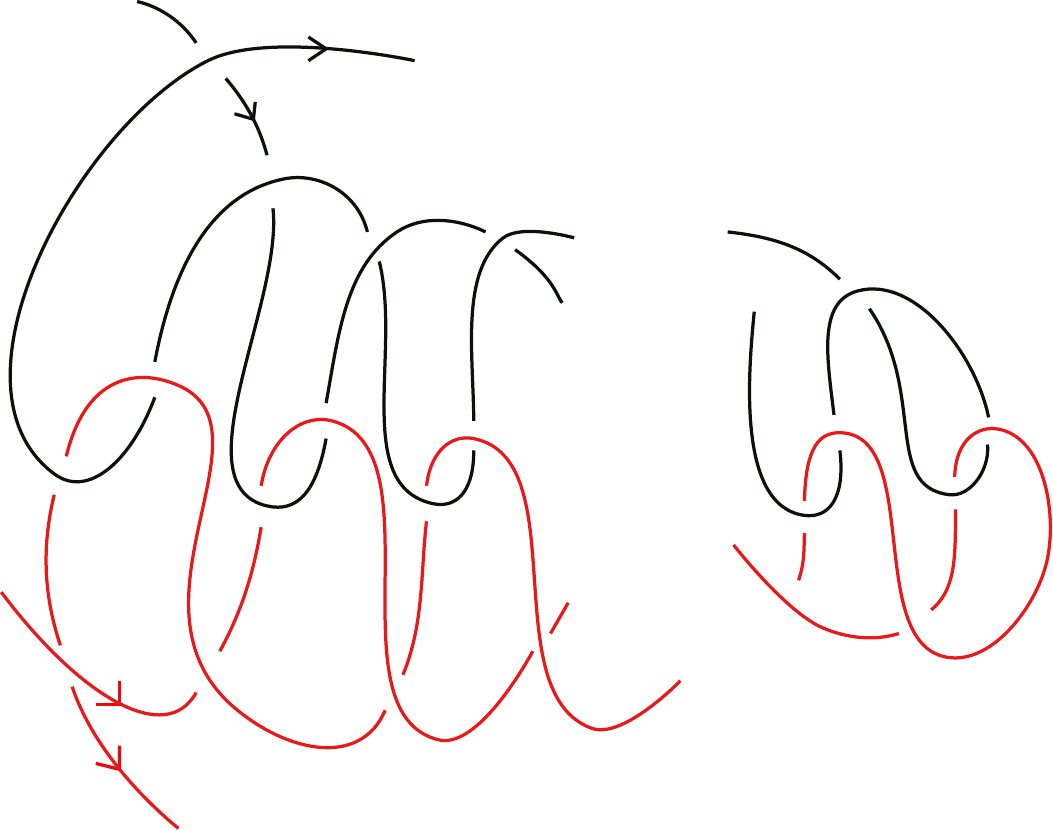}};
        \node at (4,1.5) {$\cdots$};
        \node at (4,3.5) {$\cdots$};
        \node at (0.8,5.3) {$\cK(f)$};
        \node at (5,0.6) {$\cK(g)$};
        \end{tikzpicture}
        \caption{The link $\cL(l_{*}(f,g))$ locally.}
        \label{fig:positive bridge link}
\end{figure}

\begin{figure}[H]
        \centering
        \begin{tikzpicture}
        \node[anchor=south west,inner sep=0] at (0,0){\includegraphics[width=0.5\textwidth]{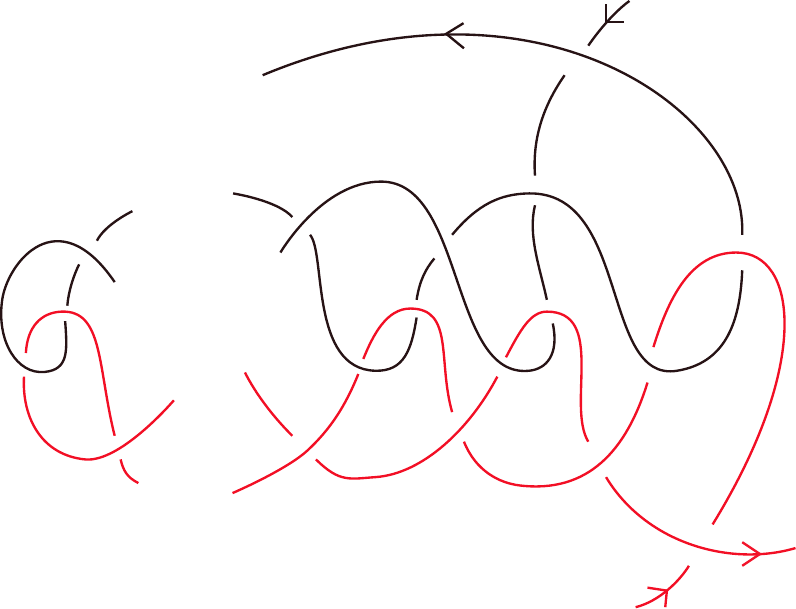}};
        \node at (1.5,3.25) {$\cdots$};
        \node at (1.55,0.95) {$\cdots$};
        \node at (4,5) {$\cK(f)$};
        \node at (4,0.2) {$\cK(g)$};
        \end{tikzpicture}
        \caption{The link $\cL(l_{-*}(f,g))$ locally.}
        \label{fig:negative bridge link}
\end{figure}

\begin{proposition}
    Given $f,g\in F_3$, and $n\in\N$, the link $\cL(l_n(f,g))$ is such that 
    \[
        \lk(\cK(f), \cK(g))=n.
    \]
    In particular, we get the following local diagram.
    
    \begin{figure}[H]
    \centering
    \begin{tikzpicture}
        \node[anchor=south west,inner sep=0] at (0,0){\includegraphics[width=0.45\textwidth]{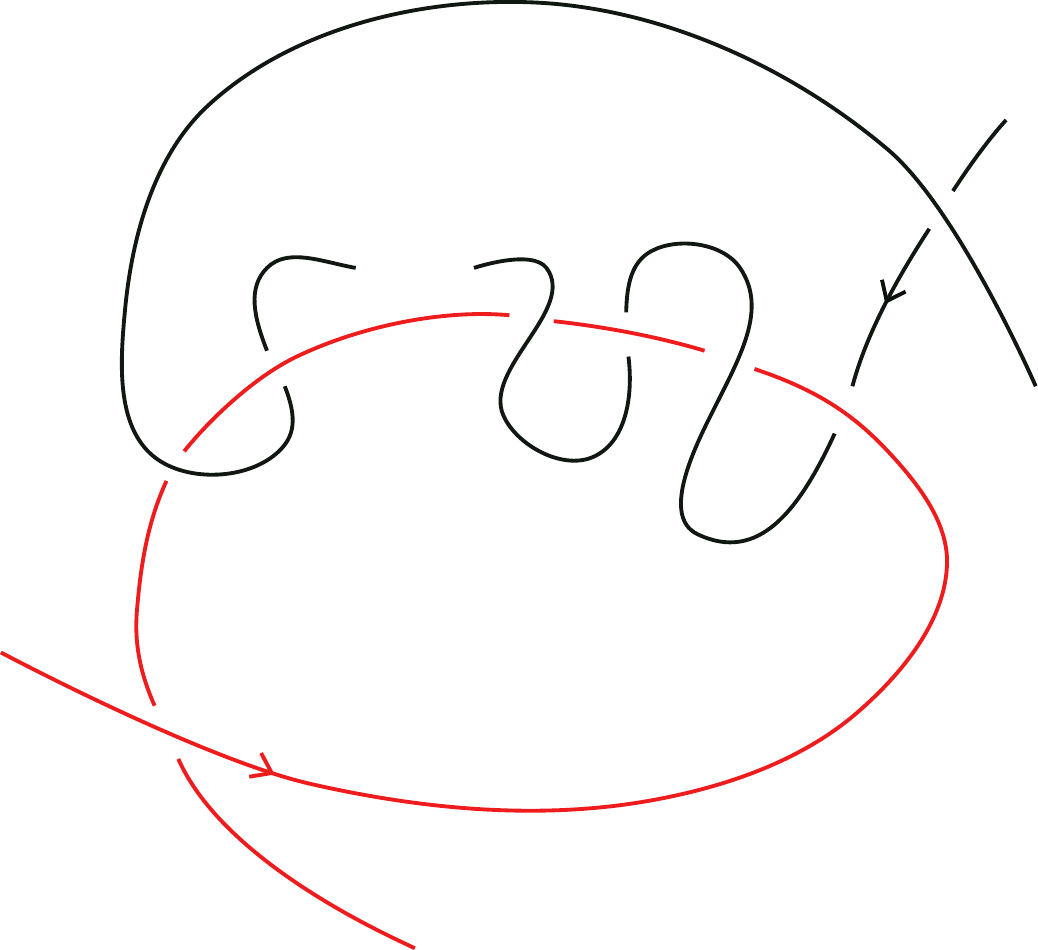}};
        \node at (2.35,3.8) {$\cdots$};
        \node at (0.9,5.3) {$\cK(f)$};
        \node at (5,0.7) {$\cK(g)$};
    \end{tikzpicture}
    \label{fig:linking two components general}
    \end{figure}
    
    Similarly, for any $n\in\N$, the link $\cL(l_{-n}(f,g))$ has
    \[
        \lk(\cK(f), \cK(g))=-n,
    \]
    and an analogous local diagram.
\end{proposition}
\begin{proof}
    We proceed by induction on $n$. We include the first three base cases for clarity.
    \begin{itemize}
        \item For $n=1$, applying Jones' algorithm, one sees that the local diagram for $\cL(l_+(f,g))$ is given by
        \begin{figure}[H]
        \centering
        \begin{tikzpicture}
        \node[anchor=south west,inner sep=0] at (0,0){\includegraphics[width=0.2\linewidth]{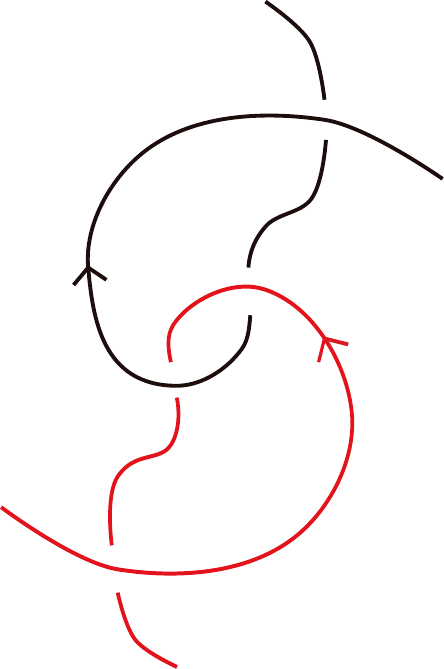}};
        \node at (0,3.2) {$\cK(f)$};
        \node at (2.7,1) {$\cK(g)$};
        \node at (3.2,2) {.};
        \end{tikzpicture}
        \end{figure}
        
        \item For $n=2$ the link $\cL(l_2(f,g))$ has local diagram as follows, where~$\sim$ denotes isotopy.

        \begin{figure}[H]
        \centering
        \begin{tikzpicture}
        \node[anchor=south west,inner sep=0] at (0,0){\includegraphics[width=0.2\linewidth]{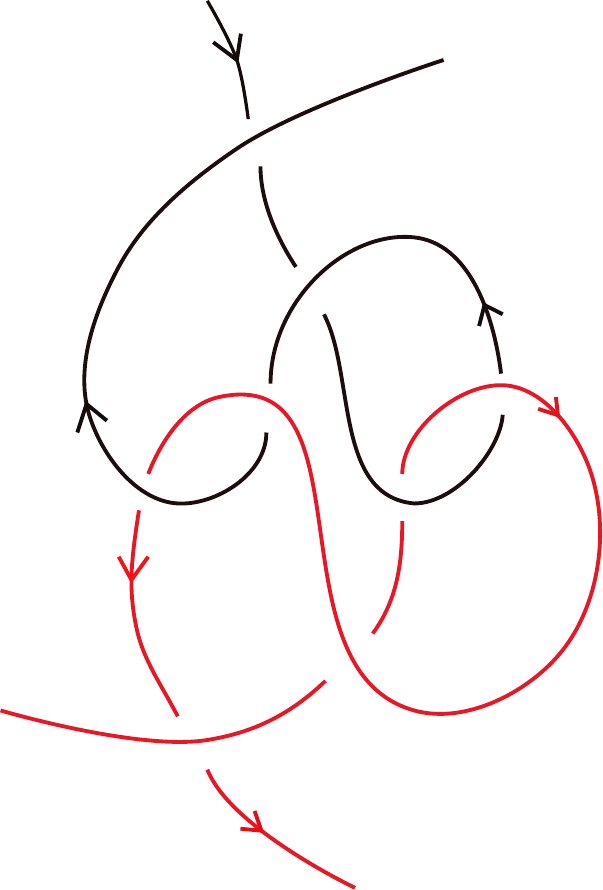}};
        \node at (3.5,2.2) {$\sim$};
        \node[anchor=south west,inner sep=0] at (4.5,0){\includegraphics[width=0.2\linewidth]{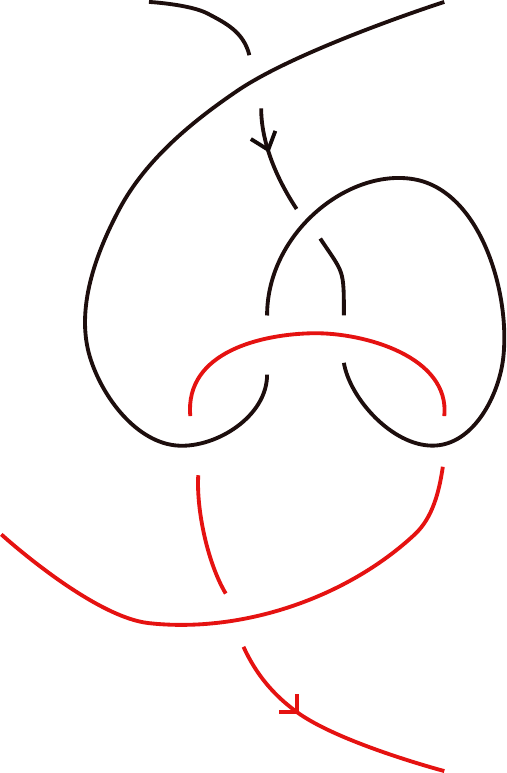}};
        \node at (8,2.2) {$\sim$};
        \node[anchor=south west,inner sep=0] at (9,0){\includegraphics[width=0.2\linewidth]{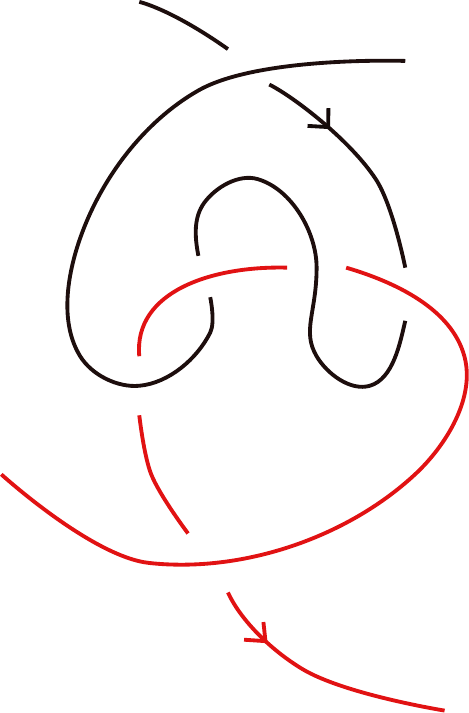}};
        \end{tikzpicture}
        \end{figure}
        
        Therefore $lk(\cK(f), \cK(g))=2$.

        \item Consider $n=3$, using the previous case we have the following local diagram.

        \begin{figure}[H]
        \centering
        \begin{tikzpicture}
        \node[anchor=south west,inner sep=0] at (0,0){\includegraphics[width=0.3\linewidth]{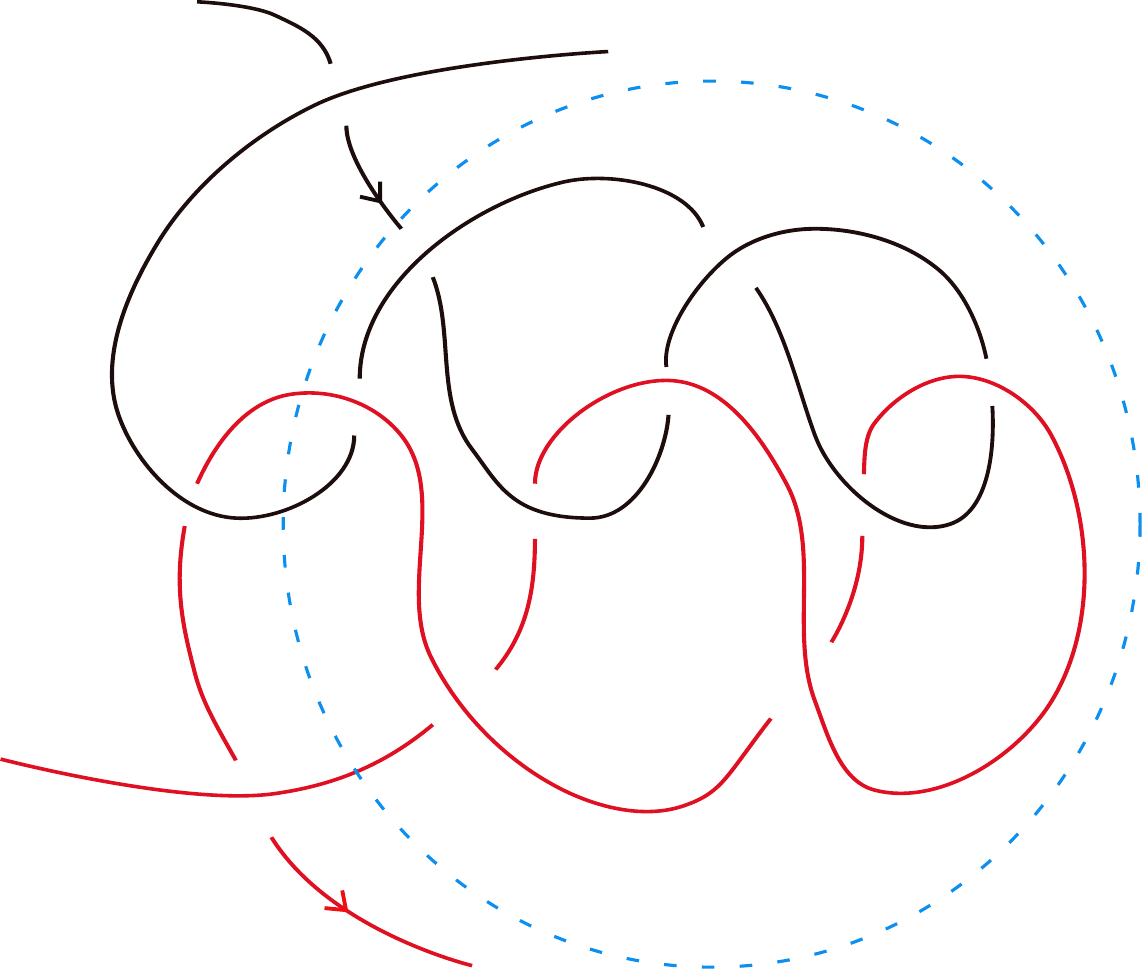}};
        \node at (5,1.5) {$\sim$};
        \node[anchor=south west,inner sep=0] at (6,0){\includegraphics[width=0.3\linewidth]{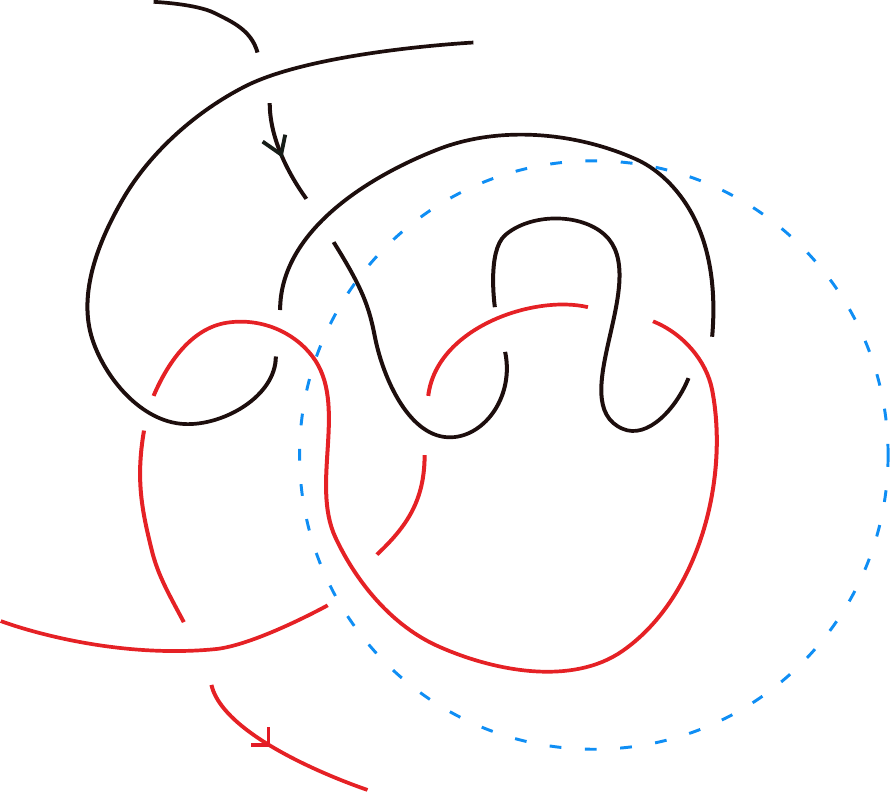}};
        \end{tikzpicture}
        \end{figure}
        
        Using Reidemeister moves, this is equivalent to

        \begin{figure}[H]
        \centering
        \begin{tikzpicture}
        \node[anchor=south west,inner sep=0] at (0,1){\includegraphics[width=0.27\linewidth]{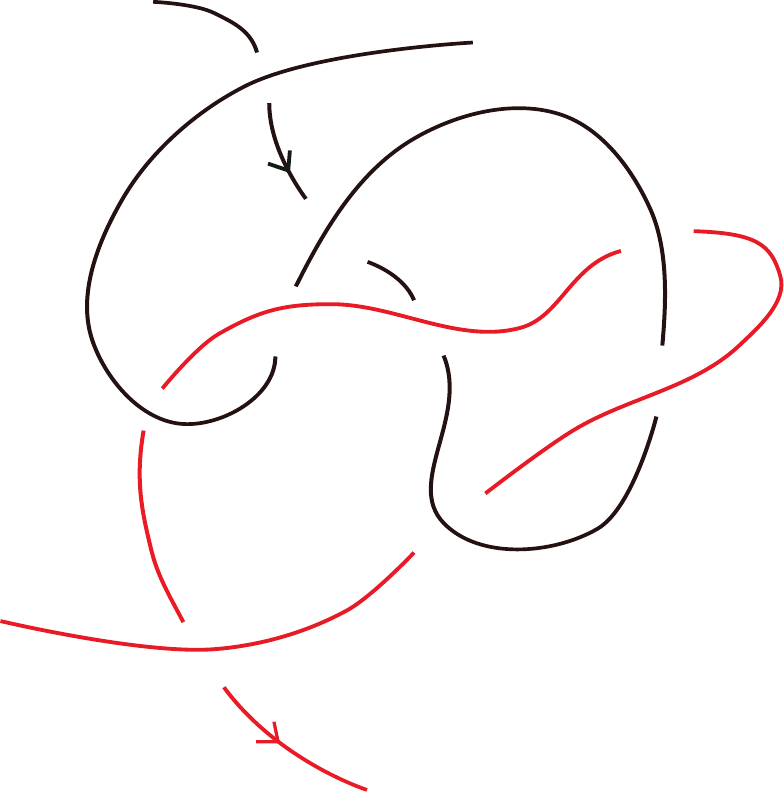}};
        \node at (4.3,2.5) {$\sim$};
        \node[anchor=south west,inner sep=0] at (5,0){\includegraphics[width=0.3\linewidth]{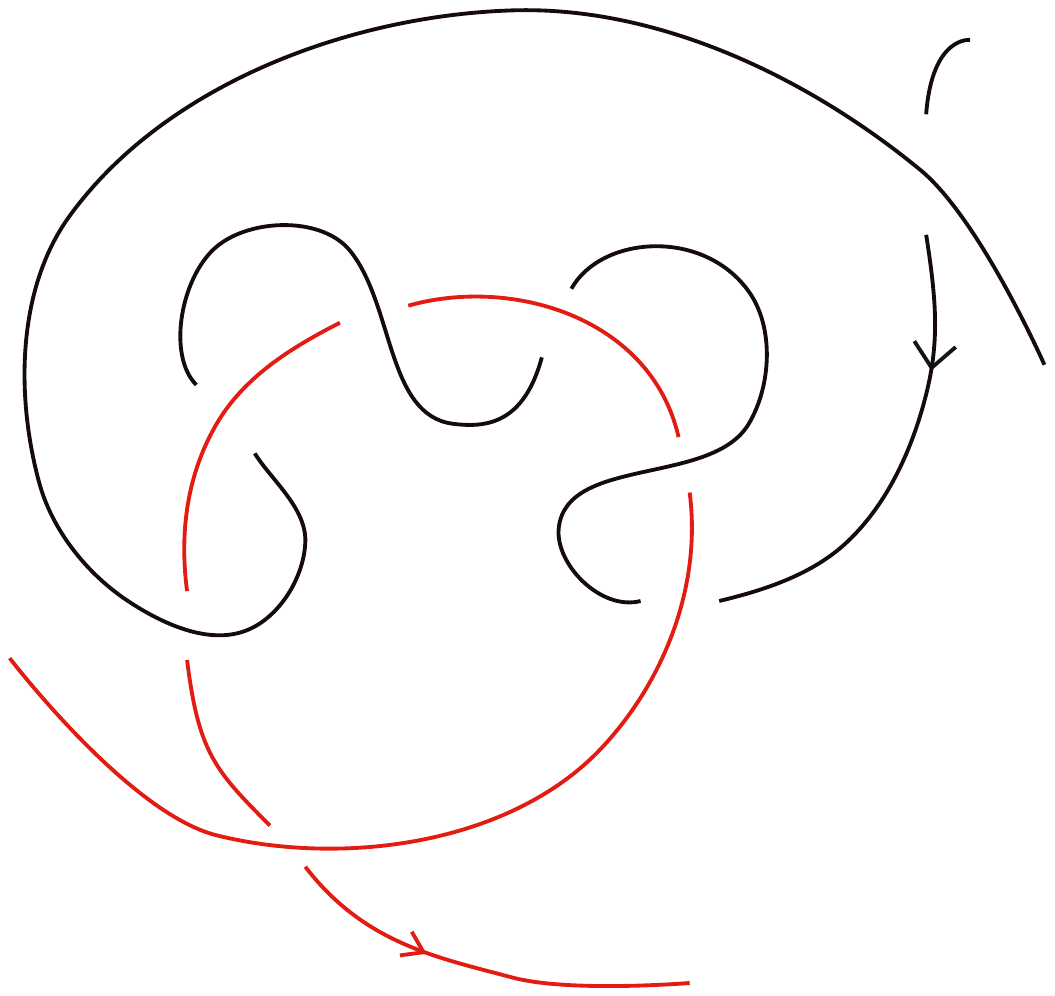}};
        \end{tikzpicture}
        \end{figure}    
        as required.
        
        \item Suppose the result holds for $n$. Consider the link $\cL(l_{n+1}(f,g))$, represented locally by:
        
        \begin{figure}[H]
        \centering
        \begin{tikzpicture}
        \node[anchor=south west,inner sep=0] at (0,0){\includegraphics[width=0.5\textwidth]{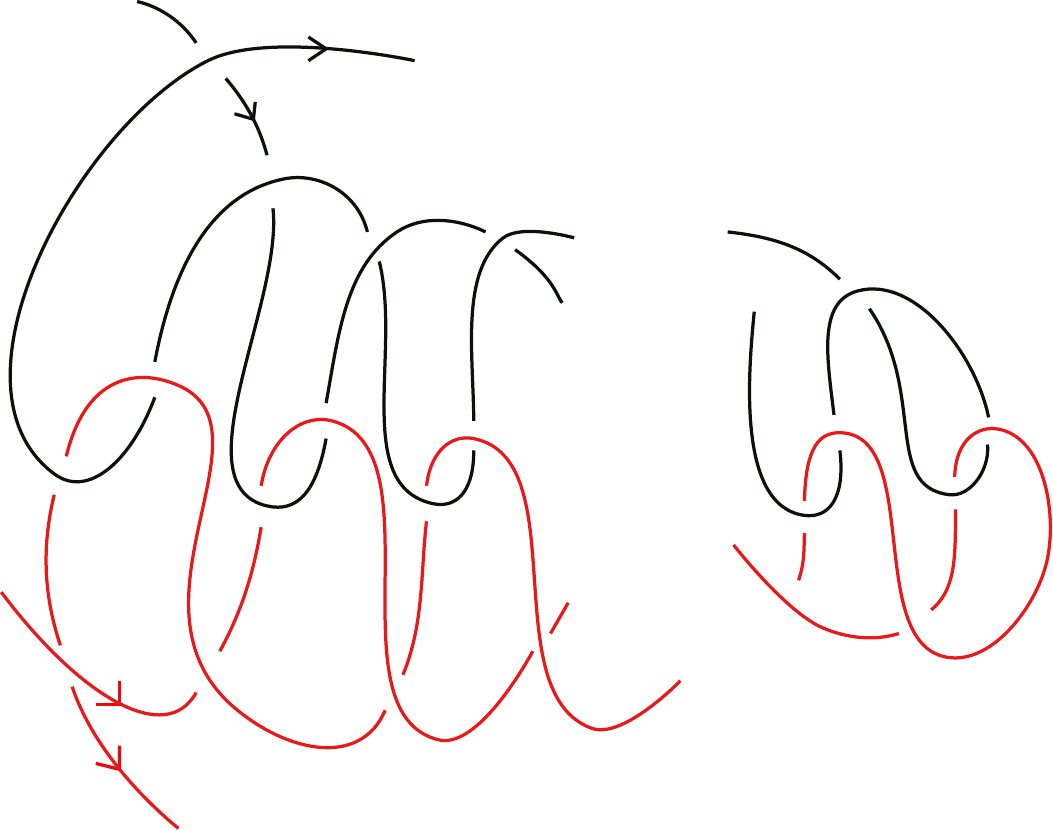}};
        \node at (4,1.5) {$\cdots$};
        \node at (3.9,1.8) {$n-3$};
        \node at (4,3.5) {$\cdots$};
        \node at (3.9,3.8) {$n-3$};
        \node at (0.8,5.3) {$\cK(f)$};
        \node at (5,0.6) {$\cK(g)$};
        \end{tikzpicture}
        \label{fig:link bn of f g}
        \end{figure}

        By the inductive hypothesis, we have
        \begin{figure}[H]
        \centering
        \begin{tikzpicture}
        \node[anchor=south west,inner sep=0] at (0,0){\includegraphics[width=0.51\textwidth]{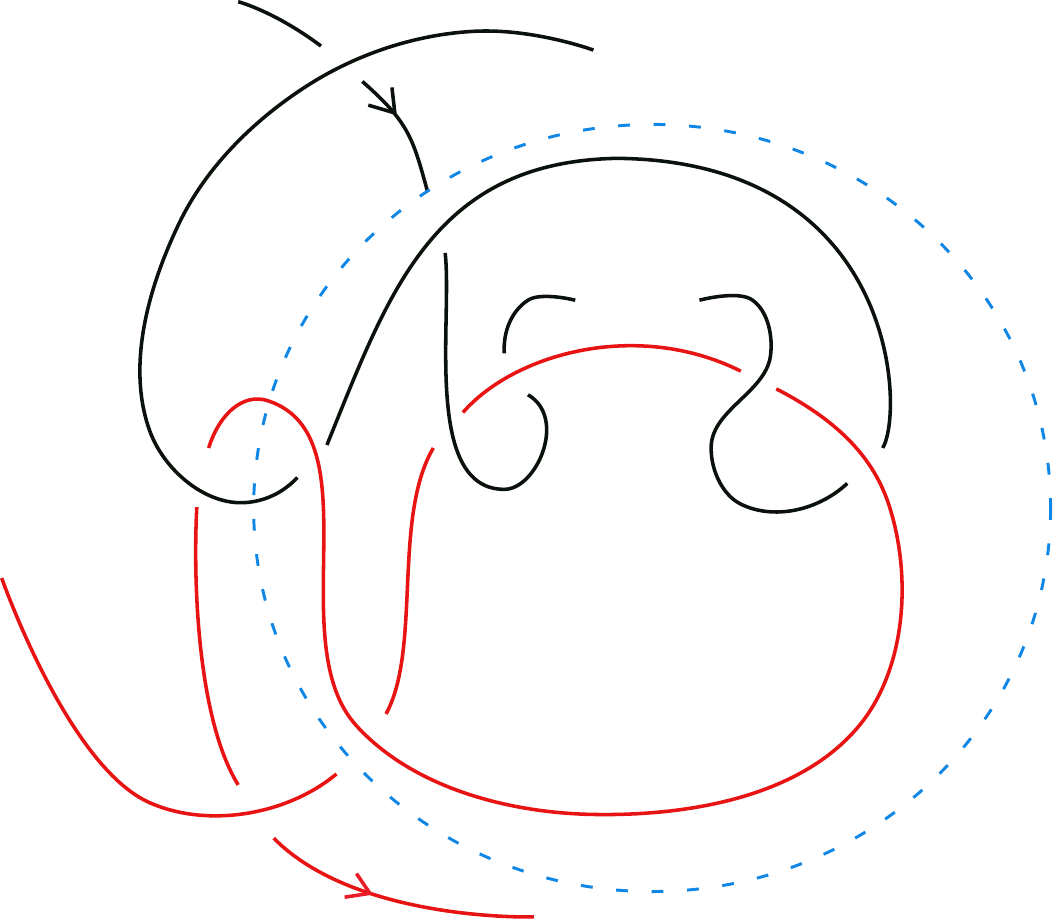}};
        \node at (3.95,3.8) {$\cdots$};
        \node at (3.95,4) {$n$};
        \node at (0.8,5.3) {$\cK(f)$};
        \node at (5.3,0.5) {$\cK(g)$};
        \end{tikzpicture}
        \end{figure}
        
        We conclude by applying Reidemeister moves to obtain the desired local diagram.

        \begin{figure}[H]
        \begin{tikzpicture}
        \node[anchor=south west,inner sep=0] at (0,0){\includegraphics[width=0.42\textwidth]{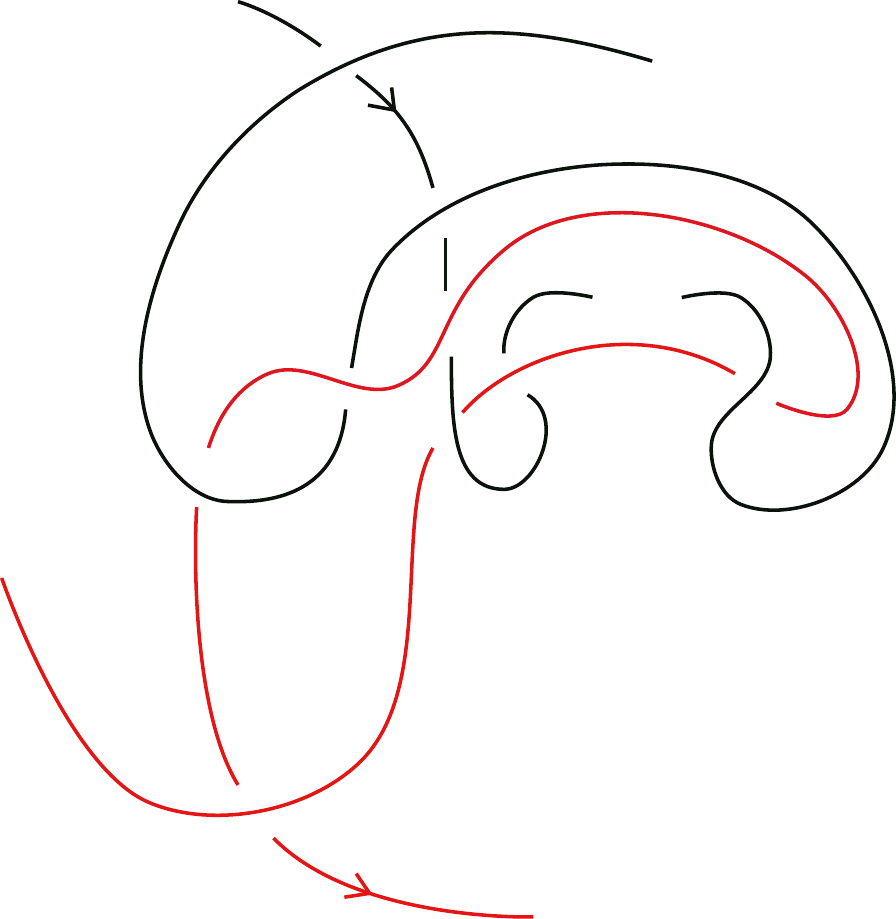}};
        \node at (3.85,3.7) {$\cdots$};
        \node at (3.85,3.9) {$n$};
        \node at (6,2.5) {$\sim$};
        \node[anchor=south west,inner sep=0] at (6.3,0){\includegraphics[width=0.42\textwidth]{img/knots/linking_proof/knot_link_general.pdf}};
        \node at (8.5,3.5) {$\cdots$};
        \end{tikzpicture}
        \qedhere
        \end{figure} 
    \end{itemize}
\end{proof}

Let $\Gamma=(V(\Gamma), E(\Gamma))$ be a finite tree endowed with a labelling on its edges, i.e.~a map $l\colon E(\Gamma)\to \Z$. Suppose $V(\Gamma)=\{v_1, \dots, v_m\}$ and let~$K_1, \dots, K_m$ be oriented knots.

\begin{definition}
\label{def:tree links}
    The \textit{tree link} $L(\Gamma; K_1, \dots, K_m)$ is defined as the link obtained by taking the disjoint union $K_1\sqcup\dots\sqcup K_m$ and, for each edge~$e_{ij}\in E(\Gamma)$, linking the components $K_i$ and $K_j$ as in \cref{fig:linking tree components}, so that $\lk(K_i,K_j)=l(e_{ij})$. 

    \begin{figure}[H]
        \begin{tikzpicture}
        \node[anchor=south west,inner sep=0] at (0,0){\includegraphics[width=0.4\textwidth]{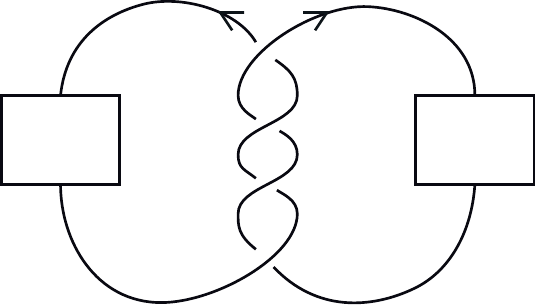}};
        \node at (0.6,1.55) {$K_i$};
        \node at (4.5,1.5) {$K_j$};
        \end{tikzpicture}
        \caption{Linking $K_i$ and $K_j$ for $l(e_{ij})=+2$.}
        \label{fig:linking tree components}
    \end{figure}   
\end{definition}

If $K_i$ is the unknot for all $i=1, \dots ,m$ we call the tree link an  \textit{iterated Hopf-tree}. A special case of iterated Hopf-trees is given by chain links, which are obtained by only taking labels $\pm 1$ -- see \cref{fig:iterated hopf}.

\begin{figure}[H]
        \begin{tikzpicture}
        \node[anchor=south west,inner sep=0] at (0.7,0.8){\includegraphics[width=0.25\textwidth]{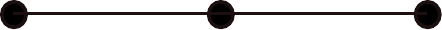}};
        \node at (0,1) {$\Gamma:$};
        \node at (0.8,0.5) {$v_2$};
        \node at (2.3,0.5) {$v_1$};
        \node at (3.8,0.5) {$v_3$};
        \node at (1.5,1.3) {$1$};
        \node at (3,1.3) {$-1$};
        \node at (4.9,1) {$\rightsquigarrow$};
        \node[anchor=south west,inner sep=0] at (5.7,0){\includegraphics[width=0.4\textwidth]{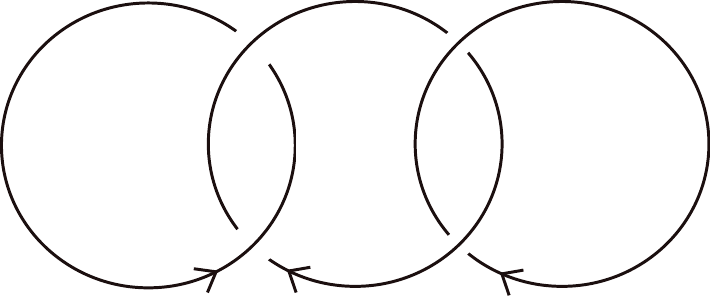}};
        \end{tikzpicture}
        \caption{A tree $\Gamma$ and the link $L(\Gamma; U,U,U)$.}
        \label{fig:iterated hopf}
\end{figure}

Using the generalised linking moves we can now construct representatives in $F_3$ for all tree links $L=L(\Gamma; K_1, \dots, K_m)$. For the interested reader, a step by step example of construction following the proof below is presented in \cref{example:tree rep}.

\begin{theorem}
\label{thm:rep tree links}
    Using $-\diamond_{\alpha}-$, $-\sqcup -$ and the linking moves, we can construct a representative $f\in F_3$ for any tree link starting from representatives of its components.
\end{theorem}

\begin{proof}
    To construct a representative $f\in F_3$, we proceed by induction on the number $m$ of components of the given tree link. Recall that, by \cref{thm:Alexander type thm}, for any $K_1, \dots, K_m$ knots there exist $g_1, \dots, g_m\in F_3$ such that $\cL(g_i)=K_i$ for all $i=1, \dots, m$. 

    If $m=1$ the tree link corresponds to a single knot $K_1$, hence we conclude by taking $f=g_1$.

    Suppose $\Gamma$ has $m$ vertices and we have constructed a representative for any tree link with up to $m-1$ components. If necessary, rename the vertices of $\Gamma$ so that $v_m$ is a leaf. Then there exists $i=1, \dots, m-1$ such that the component $K_m$ has $\lk(K_m, K_i)=n\in\Z$ and $\lk(K_m, K_j)=0$ for all $j\neq i$. 
    
    Let $L'=L(\Gamma\setminus\{v_m\};K_1, \dots, K_{m-1})$ be the tree link associated to the tree $\Gamma\setminus\{v_m\}$, with components $K_1, \dots, K_{m-1}$. By induction there exists $f'\in F_3$, written in terms of $-\diamond_{\alpha}-$, $-\sqcup-$ and linking moves of~$g_1, \dots, g_i, \dots, g_{m-1}$, such that~$\cL(f')=L'$.
    
    If $n=0$ set $f=\sqcup(f',g_m)$ and conclude. If $n\neq0$, denote by $\alpha_i$ the address of the leaf at which $g_i$ was attached in the construction  of $f'$. Let $\Tilde{f'}$ be the element obtained from $f'$ by replacing $g_i$ with $1_{F_3}$, then 
    \[
        f=\Tilde{f'}\diamond_{\alpha_i} l_n(g_i,g_m).
    \]
    The link $\cL(f)$ is, by construction, the connected sum of the unknot component $K'_i$ of $\cL(\Tilde{f'})$ and the central component of $\cL(l_n(g_i,g_m))$, i.e.~the knot $\cL(g_i)=K_i$. Hence, $\cL(f)=L(\Gamma; K_1, \dots, K_{m+1})$.
\end{proof}

The construction gives multiple possible representatives for the same link based on the ordering of the vertices of the associated tree, hence further possible Markov moves.

Note that, if $l(e_{ij})=0$ for all $i,j$, we recover the disjoint union representative $\sqcup(g_1, \dots, g_m)$.

We now give an example of representative obtained via the construction above. 

\begin{example}
\label{example:tree rep}
    Consider the tree $\Gamma$ and the corresponding tree link $L=L(\Gamma; U, U, U)$ of \cref{fig:iterated hopf}. Note that $L$ is an iterated Hopf-tree, hence we can take $g_i=y_0$ for all $i$. 

    First, consider the link associated to the single vertex $v_1$. Since $K_1=U$, we take as representative $g_1=y_0$. Consider now the tree
    \begin{figure}[H]
        \begin{tikzpicture}
        \node[anchor=south west,inner sep=0] at (1,0.8){\includegraphics[width=0.2\textwidth]{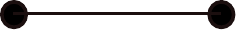}};
        \node at (0,1) {$\Gamma':$};
        \node at (1.3,0.5) {$v_2$};
        \node at (3.5,0.5) {$v_1$};
        \node at (2.4,1.3) {$1$};
        \node at (3.9,1) {,};
        \end{tikzpicture}
    \end{figure}

    then $\lk(K_2, K_1)=1$ and the representative for $L'=L(\Gamma'; U, U)$ is given by $f'=1_{F_3}\diamond l_+(y_0,y_0)$, i.e.~the element
    \begin{figure}[H]
        \begin{tikzpicture}
        \node[anchor=south west,inner sep=0] at (1,0.8){\includegraphics[width=0.3\textwidth]{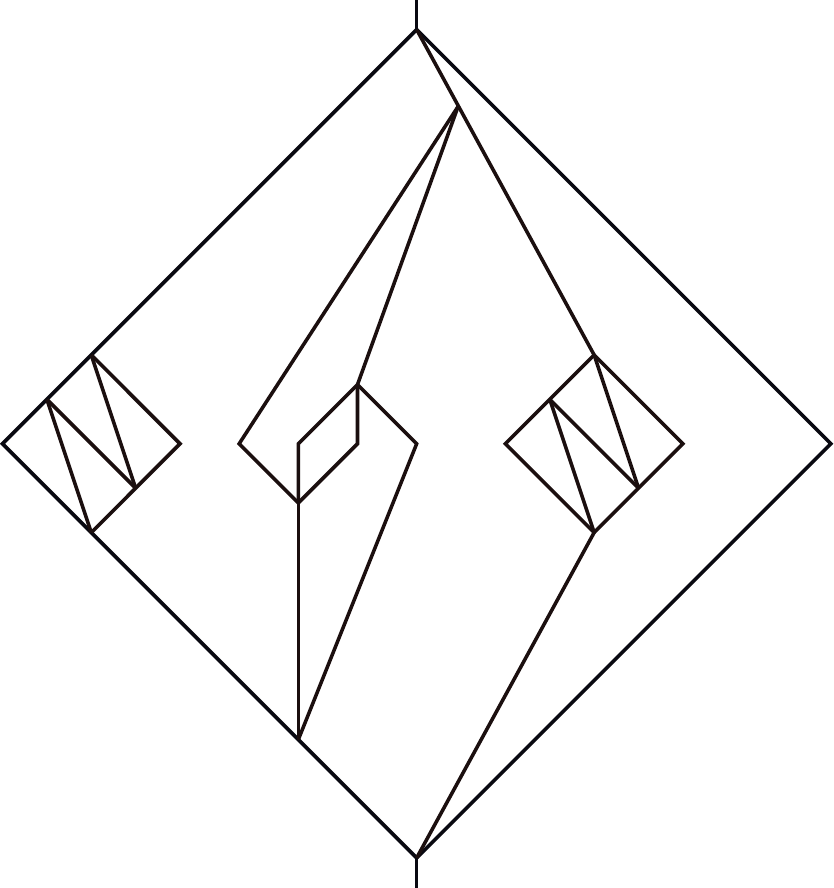}};
        \node at (0,2.9) {$f'=$};
        \node at (5.3,2.9) {.};
        \end{tikzpicture}
    \end{figure}

    Finally, take the tree $\Gamma$. The vertex $v_3$ is a leaf and the tree link~$L$ has $\lk(K_3, K_1)=-1$. In the construction of $f'$, the element $g_1$ was attached in place of the address $\alpha=12$, therefore we obtain $f$ from $f'$ as 
    \[
        f=\Tilde{f'}\diamond_{12} l_-(g_1,g_3)=\Tilde{f'}\diamond_{12} l_-(y_0,y_0),
    \]
    which corresponds to
    \begin{figure}[H]
    \centering
    \begin{tikzpicture}
    \node[anchor=south west,inner sep=0] at (0,0){\includegraphics[width=0.27\textwidth]{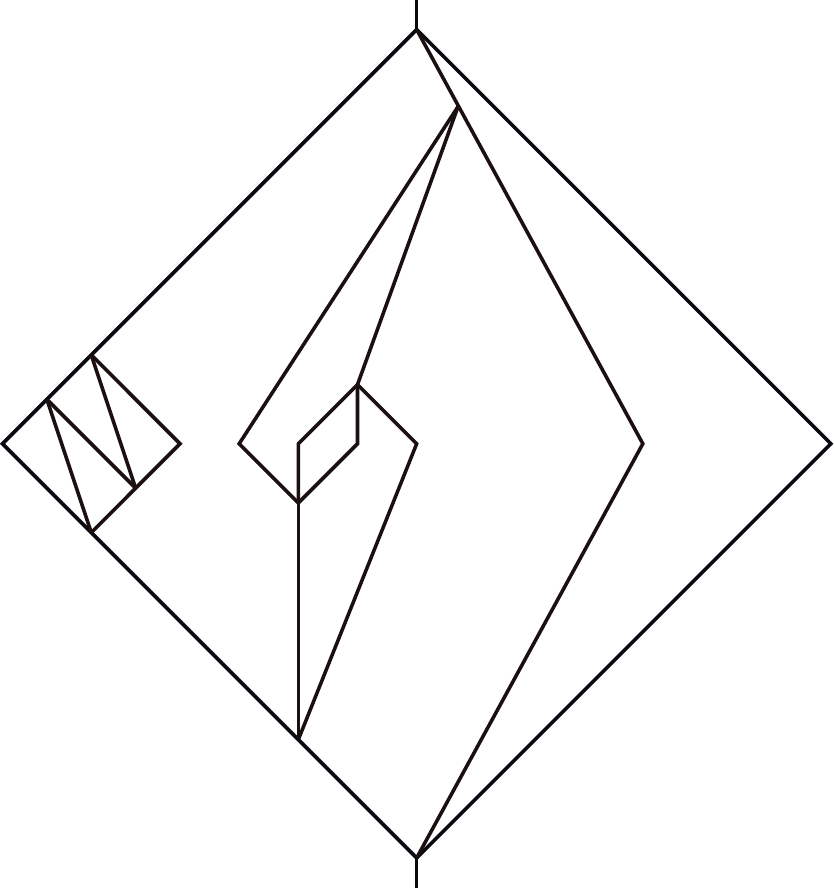}};
    \node at (3.9,1.8) {$\diamond_{12}$};
    \node[anchor=south west,inner sep=0] at (4.3,0){\includegraphics[width=0.27\textwidth]{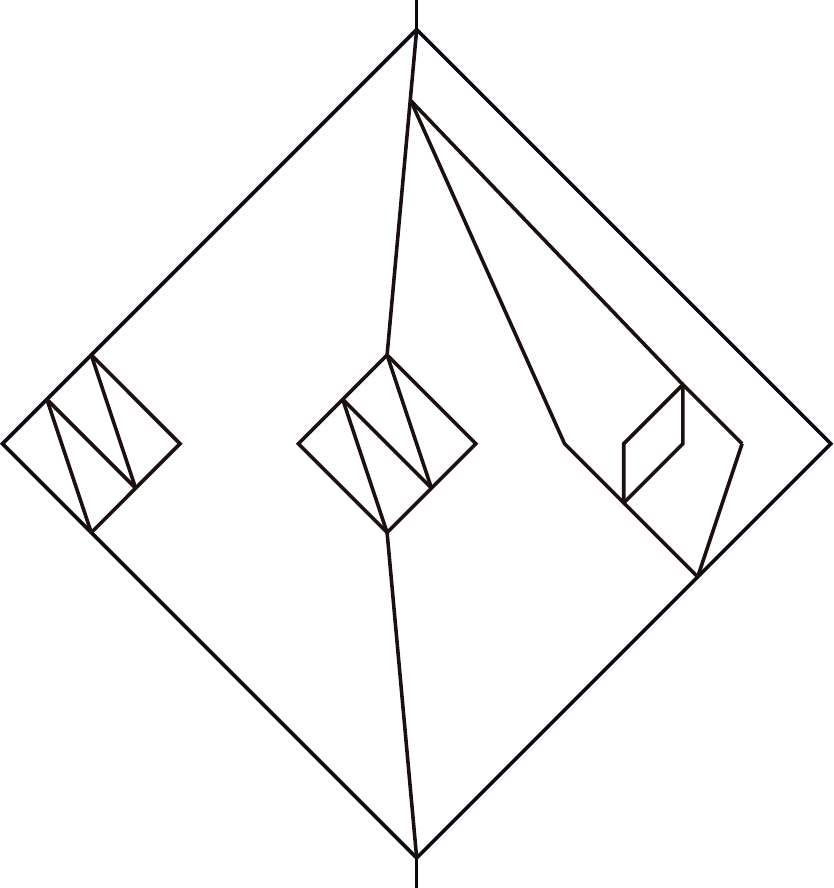}};
    \node at (8.15,1.8) {$=$};
    \node[anchor=south west,inner sep=0] at (8.6,0){\includegraphics[width=0.27\textwidth]{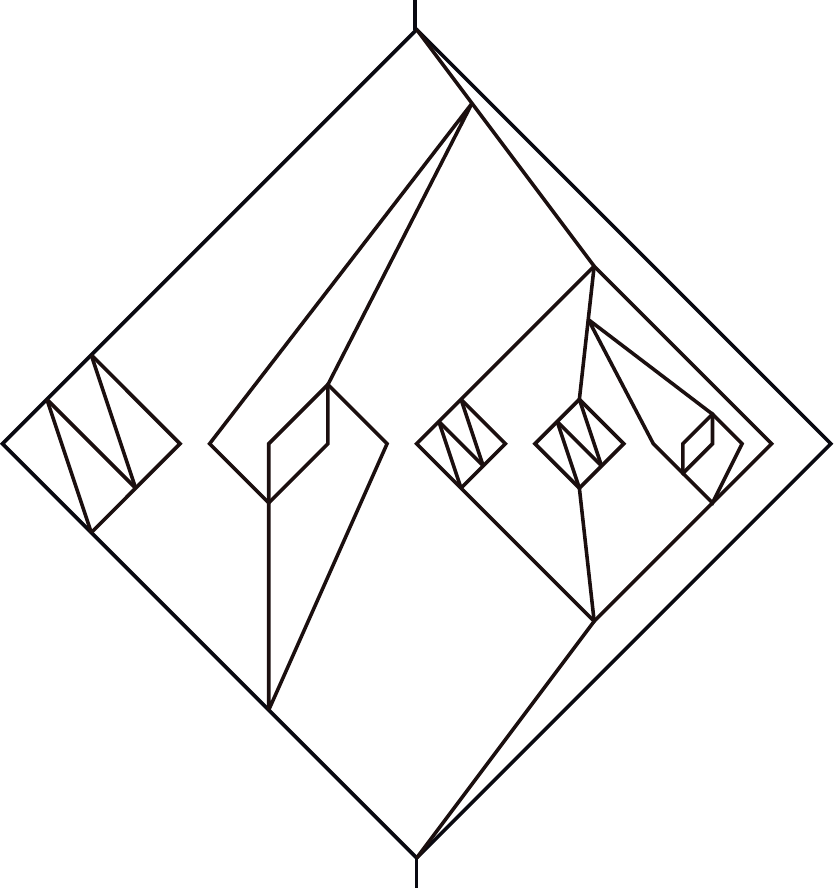}};
    \node at (12.4,1.8) {.};
    \end{tikzpicture}
    \end{figure}
    
\end{example}

We conclude with the following result combining \cref{thm:rep tree links} with prime decomposition (\cref{thm:prime decomposition}).

\begin{restate}{Theorem}{thm:tree rep}
    For all tree links one can construct a Thompson representative, up to prime components, using $-\diamond_{\alpha}-$, $-\sqcup -$ and the linking moves.
\end{restate}